\documentclass[10pt, letterpaper]{article}
\usepackage[latin1]{inputenc}
\usepackage[T1]{fontenc}
\usepackage{amsmath,amssymb,amstext}
\usepackage{theorem}
\usepackage{multicol}
\usepackage[english]{babel}
\usepackage{datetime}
\usepackage{enumerate}
\usepackage{eufrak}
\usepackage{graphicx}
\usepackage{caption}
\usepackage{subcaption}
\usepackage{float}
\usepackage{multirow}
\usepackage{color}
\usepackage{mathtools}

\graphicspath{ {./images/} }

\newtheorem{theorem}{Theorem}

\newtheorem{corollary}[theorem]{Corollary}

\newtheorem{lemma}[theorem]{Lemma}

\newtheorem{proposition}[theorem]{Proposition}
\newtheorem{remark}[theorem]{Remark}

\newenvironment{proof}[1][Proof]{\textbf{#1.} }{\ \rule{0.5em}{0.5em}}

\renewcommand{\geq}{\geqslant}

\def\1{{\mathbf{1}}}

\def\1{{\mathbf{1}}}
\def\0.5{{\frac{1}{2}}}

\input{tcilatex}

\begin{document}

\title{\textbf{Asymptotics of Yule's nonsense correlation for
Ornstein-Uhlenbeck paths: The correlated case.}}
\author{Soukaina Douissi \thanks{%
Cadi Ayyad University, UCA, National School of Applied Sciences of Marrakech
(ENSAM), BP 575, Avenue Abdelkrim Khattabi, 40000, Guéliz, Marrakech,
Morocco. Email:\texttt{s.douissi@uca.ac.ma}}, Philip Ernst\thanks{%
Imperial College London, Department of Mathematics. E-mail: \texttt{%
p.ernst@imperial.ac.uk}}, Frederi Viens \thanks{%
Department of Statistics, Rice University, USA. E-mail: \texttt{%
viens@rice.edu}}}
\date{\today }
\maketitle

\begin{abstract}
We study the continuous-time version of the empirical correlation
coefficient between the paths of two possibly correlated Ornstein-Uhlenbeck
processes, known as Yule's nonsense correlation for these paths. Using sharp
tools from the analysis on Wiener chaos, we establish the asymptotic
normality of the fluctuations of this correlation coefficient around its
long-time limit, which is the mathematical correlation coefficient between
the two processes. This asymptotic normality is quantified in Kolmogorov
distance, which allows us to establish speeds of convergence in the Type-II
error for two simple tests of independence of the paths, based on the
empirical correlation, and based on its numerator. An application to
independence of two observations of solutions to the stochastic heat
equation is given, with excellent asymptotic power properties using merely a
small number of the solutions' Fourier modes.
\end{abstract}

\section{Introduction and Setup}

The first purpose of this paper is to provide a detailed asymptotic study of
the empirical correlation coefficient $\rho \left( T\right) $ between two
standard Ornstein-Uhlenbeck (OU) processes $X_{1}$ and $X_{2}$ on a time
interval $[0,T]$, as the time horizon $T$ increases to infinity, where $\rho
\left( T\right) $ is defined below in (\ref{rho}). The OU paths $X_{1}$ and $%
X_{2}$ may or may not be correlated. This paper's second purpose is to use
those asymptotics to evaluate the power of independence tests based on $\rho
\left( T\right) $ itself as a test statistic, or on its constituent
components as test statistics. It is important to note from the outset that
the data available to compute these test statistics are the single pair of
paths $(X_{1},X_{2})$, not on repeated measurements of $X_{1}$ and/or $X_{2}$%
. This is why we chose to investigate increasing-horizon (large time)
asymptotics. This framework is well adapted to longitudinal obervational
studies with high-frequency observations, as can occur commonly in
environmental data, financial data, and many other areas where it is
inconvenient or impossible to work with highly repeatable designed
experiments.

The notion of empirical correlation coefficient $\rho \left( T\right) $ for
any pair of paths of continuous stochastic processes $(X_{1},X_{2})$ defined
on $[0,T]$ can be defined by analogy with the standard Pearson correlation
coefficient for these same paths observed in discrete time, e.g. at regular
time intervals. Because the paths are continuous, it is a trivial
application of standard Riemann integration that the standard Pearson
correlation coefficient for the discrete-time observations of $(X_{1},X_{2})$
converges, as the time step converges to 0, to the following continuous-time
statistic 
\begin{equation}
\rho (T):=\frac{Y_{12}(T)}{\sqrt{Y_{11}(T)}\sqrt{Y_{22}(T)}},  \label{rho}
\end{equation}%
where the random variables $Y_{ij}(T)$, $i,j=1,2$ are given via the
following Riemann integrals%
\begin{equation}
Y_{ij}(T):=\int_{0}^{T}X_{i}(u)X_{j}(u)du-T\bar{X}_{i}\left( T\right) \bar{%
X_{j}}\left( T\right) ,\text{ \ }\text{ \ }\bar{X}_{i}\left( T\right) :=%
\frac{1}{T}\int_{0}^{T}X_{i}(u)du.  \label{Y}
\end{equation}%
This classical analysis statement holds almost surely, as soon as the paths $%
(X_{1},X_{2})$ are continuous almost surely and are not constant over time.
That the denominators in (\ref{rho}) are non-zero comes an application of
Jensen's inequality where equality does not hold because the paths are not
constant. All other details, including the definition of $\rho $ in discrete
time, are omitted, since many references including several cited below, such
as \cite{ESW, DVE}, cover this topic.

The topic of independence testing for continuous-time modeled stochastic
processes, as a mathematical framework approximating time series observed
over long horizons or in high frequency, using $\rho \left( T\right) $ as
above, has gained renewed attention since Ernst, Shepp, and Wyner solved a
long-standing mathematical conjecture regarding the exact quantitative
behavior of the continuous-time version of Pearson's correlation coefficient
for independent random walks. Their paper \cite{ESW} discusses the history
of how $\rho \left( T\right) $ relates to the discrete-time classical
Pearson correlation coefficient. This paper can be consulted, along with its
references, for why $\rho \left( T\right) $ is the correct object of study,
as we claim above. In the case of random walks, their paper explains, as was
known since the 1960s, that the pair of processes $(X_{1},X_{2})$ should be
Brownian motions (Wiener processes), and their paper proves that when these
paths are independent, neither $\rho \left( T\right) $ nor its discrete
version converge to $0$, as one would expect for standard interpretations of
Pearson correlation coefficients. Rather, $\rho \left( T\right) $ is
constant in distribution, with a variance which they compute explicitly.
This is the so-called phenomenon of \textquotedblleft Yule's nonsense
correlation\textquotedblright , and indeed this appellation is a label for $%
\rho \left( T\right) $ itself. It was named after G. Udny Yule who
discovered this phenomenon empirically in 1926 in \cite{Yule}, and who had
conjectured that the variance of $\rho \left( T\right) $ should be
computable. The fact that, for random walks, as for other self-similar
processes such as fractional Brownian motions, $\rho \left( T\right) $ has a
stationary distribution as $T$ increases, was presumably well known by the
scholars who had studied Yule's nonsense correlation since the 1960s, and
most likely by Yule himself for the case of standard random walks. This fact
was not recorded in the literature until it was pointed out in the
introduction of \cite{DVE} when discussing the distinction in the behavior
of $\rho \left( T\right) $ between highly non-stationary paths like random
walks and Brownian motion on one hand, and i.i.d. data and stationary time
series and processes like OU on the other.

This brings us to the core of this paper's topic. It has been known for some
time that, when a pair of paths of times series is sufficiently stationary
(with some limits on how long their memory (auto-correlation) is), the
phenomenon of Yule's nonsense correlation does not hold: the Pearson
correlation of the pair of time series paths typically converges to their
underlying mathematical correlation, just as one would expect for i.i.d
data. This was established in continuous time for OU processes in the paper 
\cite{ERZ}: if $\left( X_{1},X_{2}\right) $ are two OU paths with
correlation $r$, then $\rho \left( T\right) \rightarrow r$ almost surely
i.e. as $T\rightarrow \infty $, and the fluctuations in this convergence are
Gaussian, i.e. a Central Limit Theorem (CLT) holds for $\sqrt{T}(\rho \left(
T\right) -r)$ as $T\rightarrow \infty $. The paper \cite{ERZ} was published
in 2025, but an arXiv version with this result in the case $r=0$ was posted
in 2022 as \cite{ERZ2}, which predates the publication of the paper \cite%
{DVE}. In the paper \cite{DVE}, the case of $r=0$ was studied in detail, and
a speed of convergence in this CLT was established, at the so-called
Berry-Esséen rate $1/\sqrt{T}$, using tools from the so-called Wiener chaos
analysis. That paper also studied the discrete high-frequency version of $%
\rho \left( T\right) $ and established a rate of convergence of its normal
fluctuations which depends on $T$ and on the rate of observations. A study
of moderate deviations for $\rho \left( T\right) $ is given in the 2025
paper \cite{ZJW}, where the OU processes are observed at discrete time and
in high frequency, similarly to the dicrete observation restrictions placed
on the OU processes in the earlier contribution \cite{DVE}. This leaves the
case of correlated paths ($r\neq 0$) in continuous time open. That question
was taken up in fully dicrete time in the preprint \cite{EH} in the context
of AR(1) processes, i.e. without simultaneous restrictions on high-frequency
observations and increasing horizon. They establish an exact distribution
theory, and they study asymptotics of the discrete-time version of the
empirical correlation, quantitatively, using basic estimates from the
Malliavin calculus, similar to the tools developed in \cite{DVE}. Since that
empirical correlation converges to the underlying mathematical correlation $r
$, the paper \cite{EH}'s distribution theory is used to prove a Berry-Essé%
en-type theorem in Kolmogorov distance for the Gaussian fluctuations of the
discrete empirical correlation. This immediately allows \cite{EH} to prove
that a simple test of independence is asymptotically powerful, similar to
what we do in the present article.

The current article picks up the framework in \cite{DVE}, in continuous
time, now allowing $(X_{1},X_{2})$ to be correlated, and taking the analysis
of independence testing further. That is also the topic of the preprint \cite%
{EH}, in discrete time, as just mentioned. This current paper compares with
the fully discrete-time setting of \cite{EH} in the following ways.
Superficially, both papers use estimates of distances between probability
measures on Wiener chaos which can be found in the work of Nourdin and
Peccati, though the current paper relies on the optimal version of these
estimates in \cite{NP2015}, while \cite{EH} works with the possibly
suboptimal estimates in the earlier research, summarized in the book \cite%
{NP-book}. The extraordinarily detailed exact-distribution theory
calcultions performed in \cite{EH} are helpful to achieving what appear to
be sharp estimates via the tools in \cite{NP-book}, which is why there does
not appear to be any downside to using those less optimal methods,
circumventing the need to perform third-cumulant calculations. In contrast,
in the current paper, as seen for example in the proof of Proposition \ref%
{cumulants} below, third- and fourth-cumulant calculations are needed to
apply the Optimal Fourth Moment theorem in \cite{NP2015}. The advantage of
using this theorem is a guarantee of optimality assuming efficient cumulant
estimations; another advantage is the avoidance of any exact distribution
theory, which significantly lightens the technicalities needed to establish
probability measure distance estimates on Wiener chaos. Another major
difference between \cite{EH} and the current paper is that the latter is in
continuous time and the former is in discrete time; this is perhaps a
superficial distinction in terms of results, since both papers concentrate
on increasing-horizon asymptotics. However, in terms of proofs, whether the
method of exact distribution theory can be applied to the continuous-time
framework is an open question. The answer could be affirmative, but it is
unclear whether the necessary technicalities are worth the effort. One could
be particularly averse to engaging in the required spectral analysis, given
how much effort and talent was expended in \cite{EH} to handle the
finite-dimensional matrix analysis needed there. In terms of applications to
testing, the current paper engages in a detailed quantitative power
analysis, proposing two different tests depending on whether one uses the
full empirical correlation coefficient, or only the covariance in its
numerator; \cite{EH} applies its Berry-Esséen result to the empirical
correlation for the power calculation, in an efficient way.

Specifically, in the remainder of this paper, $(X_{1},X_{2})$ are a pair of
two OU processes with the same known drift parameter $\theta >0$, namely $%
X_{i}$ solves the linear SDE, for $i=1,2$ 
\begin{equation}
dX_{i}(t)=-\theta X_{i}(t)dt+dW^{i}(t),\text{ \ }t\geq 0  \label{OU}
\end{equation}%
where we assume $X_{i}(0)=0$, $i=1,2$ for the sake of reducing
technicalities, where the driving noises $(W^{1}(t))_{t\geq 0}$, $%
(W^{2}(t))_{t\geq 0}$ are two standard Brownian motions (Wiener processes).
As mentioned, this paper builds a statistical test of independence (or
dependence) of the pair of OU processes $(X_{1},X_{2})$ using $\rho \left(
T\right) $ for large $T$. That is, we propose a test for the following null
hypothesis

\begin{center}
$H_0 :$ $(X_1)$ and $(X_2)$ are independent.

Versus the Alternative Hypothesis

$H_{a}:$ $(X_{1})$ and $(X_{2})$ are correlated with some fixed $%
r=cor(W_{1},W_{2})\in \lbrack -1,1]\backslash \{0\}.$
\end{center}

The reader may note that this is a simple hypothesis test, in the sense that
the alternative hypothesis is specific to a fixed value $r\neq 0.$ Because
of the infinite-dimensional nature of the objects of study, we believe that
a more general hypothesis test, such as a full two-sided test where the
alternative covers all non-zero values of $r$, would not be asymptotically
powerful. For this reason, we do not consider such broader alternatives.

As mentioned, under $H_{0}$, by exploiting the second-Wiener-chaos
properties of the three random variables $%
(Y_{i,j}(T),(i,j)=(1,1),(2,2),(1,2))$ appearing as components of the ratio $%
\rho (T)$ in (\ref{rho}), the paper \cite{DVE} shows, using the connection
between the Malliavin Calculus and Stein's method, that the speed of
convergence in law of $T^{1/2}\rho (T)$ to the normal law $\mathcal{N}%
(0,1/\theta )$ in the Kolmogorov distance $d_{Kol}$ is bounded above by a
log-corrected Berry-Ess en rate $T^{-1/2}\log (T)$.

Therefore, as a first step in looking for an asymptotically powerful test to
reject the null hyothesis of independence, we will study the Gaussian
fluctuations for the statistic $\rho (T)$ under $H_{a}$. This is the topic
of Section \ref{FLUCTU}. We follow that with Section \ref{TESTING} where we
identify an asymptotically powerful test for rejecting the null. Finally,
section \ref{SPDE} provides an interesting example of what Section \ref%
{TESTING} implies in the case of stochastic differential equations in
infinite dimensions, namely how to build a test of independence for
solutions of the stochastic heat equation. But first, in Section \ref{Wiener}%
, we begin with some preliminary information on analysis on Wiener space, to
help make this paper essentially self-contained beyond the construction of
basic objects like the Wiener process.

\section{Elements of the analysis on Wiener space\label{Wiener}}

This section provides essential facts from basic probability, the Malliavin
calculus, and more broadly the analysis on Wiener space. These facts and
their corresponding notations underlie all the results of this paper. This
is because, as mentioned in the introduction, and as noted in the paper \cite%
{DVE}, the three constituent components of $\rho (T)$ involve random
variables in the so-called second Wiener chaos. We have strived to make this
section self contained and logically articulated, presenting material needed
to understand all technical details in this paper, and elements that help
appreciate how these background results fit together as part of the analysis
on Wiener space.

Of particular importance below, when performing exact calculations on these
variables, are the isometry and product formula on Wiener chaos. Another
important property of Wiener chaos explained below and used in this paper is
the so-called hypercontractivity, or equivalence of norms on Wiener chaos.
The crux of the quantitative arguments we make in this paper, to estimate
the rate of normal fluctuations for $\rho (T)$ and its components, come from
the so-called optimal fourth moment theorem on Wiener chaos, also explained
in detail below. It is the precision afforded by that theorem that allows us
to produce tests of independence with good, quantitative properties of
asymptotic power. That theorem, as explained below, supercedes a previous
theorem known as the fourth moment theorem, which we also present below,
along with related results about the connection between Stein's method and
Malliavin derivatives, to give the full context of how all these techniques
fit together. Strictly speaking, the original fourth moment theorem, and the
connection between Malliavin derivatives and Stein's method, are not used
directly in the current paper, but we include them in this section's
didactic overview because we believe omitting them would not be helpful to
readers who have some familiarity with some of the tools but not others. The
interested reader can find more details about the results in this section by
consulting the books \cite[Chapter 1]{nualart-book} and \cite[Chapter 2]%
{NP-book}. However, the details of the optimal fourth moment theorem should
be consulted in the original article \cite{NP2015}.

With $\left( \Omega ,\mathcal{F},\mathbf{P}\right) $ denoting the Wiener
space of a standard Wiener process $W$, for a deterministic function $h\in
L^{2}\left( \mathbf{R}_{+}\right) =:{{\mathcal{H}}}$, the Wiener integral $%
\int_{\mathbf{R}_{+}}h\left( s\right) drW\left( s\right) $ is also denoted
by $W\left( h\right) $. The inner product $\int_{\mathbf{R}_{+}}f\left(
s\right) g\left( s\right) ds$ will be denoted by $\left\langle
f,g\right\rangle _{{\mathcal{H}}}$.

\begin{itemize}
\item \textbf{The Wiener chaos expansion}. For every $q\geq 1$, ${\mathcal{H}
}_{q}$ denotes the $q$th Wiener chaos of $W$, defined as the closed linear
subspace of $L^{2}(\Omega )$ generated by the random variables $%
\{H_{q}(W(h)),h\in {{\mathcal{H}}},\Vert h\Vert _{{\mathcal{H}}}=1\}$ where $%
H_{q}$ is the $q$th Hermite polynomial. Wiener chaos of different orders are
orthogonal in $L^{2}\left( \Omega \right) $. The so-called Wiener chaos
expansion is the fact that any $X\in L^{2}\left( \Omega \right) $ can be
written as 
\begin{equation}
X=\mathbf{E}X+\sum_{q=0}^{\infty }X_{q}  \label{WienerChaos}
\end{equation}
for some $X_{q}\in {\mathcal{H}}_{q}$ for every $q\geq 1$. This is
summarized in the direct-orthogonal-sum notation $L^{2}\left( \Omega \right)
=\oplus _{q=0}^{\infty }{\mathcal{H}}_{q}$. Here ${\mathcal{H}}_{0}$ denotes
the constants.

\item \textbf{Relation with Hermite polynomials. Multiple Wiener integrals}.
The mapping ${I_{q}(h^{\otimes q}):}=q!H_{q}(W(h))$ is a linear isometry
between the symmetric tensor product ${\mathcal{H}}^{\odot q}$ space of
functions on $\left( \mathbf{R}_{+}\right) ^{q}$ (equipped with the modified
norm $\Vert .\Vert _{{\mathcal{H}}^{\odot q}}=\sqrt{q!}\Vert .\Vert _{{%
\mathcal{H}}^{\otimes q}}$) and the $q$th Wiener chaos space ${\mathcal{H}}%
_{q}$ . To relate this to standard stochastic calculus, one first notes that 
${\ I_{q}(h^{\otimes q})}$ can be interpreted as the multiple Wiener
integral of ${h^{\otimes q}}$ w.r.t. $W$. By this we mean that the
Riemann-Stieltjes approximation of such an integral converges in $%
L^{2}\left( \Omega \right) $ to ${I_{q}(h^{\otimes q})}$. This is an
elementary fact from analysis on Wiener space, which can also be proved
using standard stochastic calculus for square-integrable martingales,
because the multiple integral interpretation of ${I_{q}(h^{\otimes q})}$ as
a Riemann-Stieltjes integral over $\left( \mathbf{R}_{+}\right) ^{q}$ can be
further shown to coincide with $q!$ times the iterated It integral over the
first simplex in $\left( \mathbf{R}_{+}\right) ^{q}$.

More generally, for $X$ and its Wiener chaos expansion (\ref{WienerChaos})
above, each term $X_{q}$ can be interpreted as a multiple Wiener integral $%
I_{q}\left( f_{q}\right) $ for some $f_{q}\in {\mathcal{H}}^{\odot q}$.

\item \textbf{The product formula - Isometry property}. For every $f,g\in {{%
\ \mathcal{H}}}^{\odot q}$ the following extended isometry property holds 
\begin{equation}
E\left( I_{q}(f)I_{q}(g)\right) =q!\langle f,g\rangle _{{\mathcal{H}}%
^{\otimes q}}.  \label{isometry}
\end{equation}%
Similarly as for ${I_{q}(h^{\otimes q})}$, this formula is established using
basic analysis on Wiener space, but it can also be proved using standard
stochastic calculus, owing to the coincidence of $I_{q}(f)$ and $I_{q}(g)$
with iterated It integrals. To do so, one uses It 's version of integration
by parts, in which iterated calculations show coincidence of the expectation
of the bounded variation term with the right-hand side above. What is
typically referred to as the Product Formula on Wiener space is the version
of the above formula before taking expectations (see \cite[Section 2.7.3]%
{NP-book}). In our work, beyond the zero-order term in that formula, which
coincides with the expectation above, we will only need to know the full
formula for $q=1$, which is: 
\begin{equation}
I_{1}(f)I_{1}(g)=2^{-1}I_{2}\left( f\otimes g+g\otimes f\right) +\langle
f,g\rangle _{{\mathcal{H}}}.  \label{product-formula}
\end{equation}

\item \textbf{Hypercontractivity in Wiener chaos}. For $h\in {\mathcal{H}}%
^{\otimes q}$, the multiple Wiener integrals $I_{q}(h)$, which exhaust the
set ${\mathcal{H}}_{q}$, satisfy a hypercontractivity property (equivalence
in ${\mathcal{H}}_{q}$ of all $L^{p}$ norms for all $p\geq 2$), which
implies that for any $F\in \oplus _{l=1}^{q}{\mathcal{H}}_{l}$ (i.e. in a
fixed sum of Wiener chaoses), we have 
\begin{equation}
\left( E\big[|F|^{p}\big]\right) ^{1/p}\leqslant c_{p,q}\left( E\big[|F|^{2}%
\big]\right) ^{1/2}\ \mbox{ for any }p\geq 2.  \label{hypercontractivity}
\end{equation}%
It should be noted that the constants $c_{p,q}$ above are known with some
precision when $F$ is a single chaos term: indeed, by Corollary 2.8.14 in 
\cite{NP-book}, $c_{p,q}=\left( p-1\right) ^{q/2}$.

\item \textbf{Malliavin derivative}. The Malliavin derivative operator $D$
on Wiener space is not needed explicitly in this paper. However, because of
the fundamental role $D$ plays in evaluating distances between random
variables, it is helpful to introduce it, to justify the estimates (\ref%
{NPobschaos}) and (\ref{optimal berry esseen}) below. For any univariate
function $\Phi \in C^{1}\left( \mathbf{R}\right) $ with bounded derivative,
and any $h\in {\ \mathcal{H}}$, the Malliavin derivative of the random
variable $X:=\Phi \left( W\left( h\right) \right) $ is defined to be
consistent with the following chain rule: 
\begin{equation*}
DX:X\mapsto D_{r}X:=\Phi ^{\prime }\left( W\left( h\right) \right) h\left(
r\right) \in L^{2}\left( \Omega \times \mathbf{R}_{+}\right) .
\end{equation*}%
A similar chain rule holds for multivariate $\Phi $. One then extends $D$ to
the so-called Gross-Sobolev subset $\mathbf{D}^{1,2}\varsubsetneqq
L^{2}\left( \Omega \right) $ by closing $D$ inside $L^{2}\left( \Omega
\right) $ under the norm defined by its square%
\begin{equation*}
\left\Vert X\right\Vert _{1,2}^{2}:=\mathbf{E}\left[ X^{2}\right] +\mathbf{E}%
\left[ \int_{\mathbf{R}_{+}}\left\vert D_{r}X\right\vert ^{2}dr\right] .
\end{equation*}%
All Wiener chaos random variable are in the domain $\mathbf{D}^{1,2}$ of $D$%
. In fact this domain can be expressed explicitly for any $X$ as in (\ref%
{WienerChaos}): $X\in \mathbf{D}^{1,2}$ if and only if $\sum_{q}qq!\Vert
f_{q}\Vert _{{\mathcal{H}}^{\otimes q}}^{2}<\infty $.

\item \textbf{Generator }$L$ \textbf{of the Ornstein-Uhlenbeck semigroup}.
The linear operator $L$ is defined as being diagonal under the Wiener chaos
expansion of $L^{2}\left( \Omega \right) $: ${\mathcal{H}}_{q}$ is the
eigenspace of $L$ with eigenvalue $-q$, i.e. for any $X\in {\mathcal{H}}_{q}$
, $LX=-qX$. We have $Ker$($L)=$ ${\mathcal{H}}_{0}$, the constants. The
operator $-L^{-1}$ is the negative pseudo-inverse of $L$, so that for any $%
X\in {\mathcal{H}}_{q}$, $-L^{-1}X=q^{-1}X$. Since the variables we will be
dealing with in this article are finite sums of elements of ${\mathcal{H}}%
_{q}$, the operator $-L^{-1}$ is easy to manipulate thereon. The use of $L$
is crucial when evaluating the total variation distance between the laws of
random variables in Wiener chaos, as we will see shortly.

\item \textbf{Distances between random variables}. The following is
classical. If $X,Y$ are two real-valued random variables, then the total
variation distance between the law of $X$ and the law of $Y$ is given by 
\begin{equation*}
d_{TV}\left( X,Y\right) :=\sup_{A\in \mathcal{B}({\mathbb{R}})}\left\vert P%
\left[ X\in A\right] -P\left[ Y\in A\right] \right\vert
\end{equation*}%
where the supremum is over all Borel sets. The Kolmogorov distance $%
d_{Kol}\left( X,Y\right) $ is the same as $d_{TV}$ except one take the sup
over $A$ of the form $(-\infty ,z]$ for all real $z$. The Wasserstein
distance uses Lipschitz rather than indicator functions: 
\begin{equation*}
d_{W}\left( X,Y\right) :=\sup_{f\in Lip(1)}\left\vert Ef(X)-Ef(Y)\right\vert
,
\end{equation*}%
$Lip(1)$ being the set of all Lipschitz functions with Lipschitz constant $%
\leqslant 1$.

\item \textbf{Malliavin operators and distances between laws on Wiener space}%
. There are two key estimates linking total variation distance and the
Malliavin calculus, which were both obtained by Nourdin and Peccati. The
first one is an observation relating an integration-by-parts formula on
Wiener space with a classical result of Ch. Stein. The second is a
quantitatively sharp version of the famous 4th moment theorem of Nualart and
Peccati. Let $N$ denote the standard normal law.

\begin{itemize}
\item \textbf{The observation of Nourdin and Peccati}. Let $X\in \mathbf{D}%
^{1,2}$ with $\mathbf{E}\left[ X\right] =0$ and $Var\left[ X\right] =1$.
Then (see \cite[Proposition 2.4]{NP2015}, or \cite[Theorem 5.1.3]{NP-book}),
for $f\in C_{b}^{1}\left( \mathbf{R}\right) $, 
\begin{equation*}
E\left[ Xf\left( X\right) \right] =E\left[ f^{\prime }\left( X\right)
\left\langle DX,-DL^{-1}X\right\rangle _{\mathcal{H}}\right]
\end{equation*}%
and by combining this with properties of solutions of Stein's equations, one
gets 
\begin{equation}
d_{TV}\left( X,N\right) \leqslant 2E\left\vert 1-\left\langle
DX,-DL^{-1}X\right\rangle _{\mathcal{H}}\right\vert .  \label{NPobs}
\end{equation}%
One notes in particular that when $X\in {\mathcal{H}}_{q}$, since $%
-L^{-1}X=q^{-1}X$, we obtain conveniently 
\begin{equation}
d_{TV}\left( X,N\right) \leqslant 2E\left\vert 1-q^{-1}\left\Vert
DX\right\Vert _{\mathcal{H}}^{2}\right\vert .  \label{NPobschaos}
\end{equation}%
It is this last observation which leads to a quantitative version of the 
\emph{fourth moment theorem} of Nualart and Peccati, which entails using
Jensen's inequality to note that the right-hand side of (\ref{NPobs}) is
bounded above by the variance of $\left\langle DX,-DL^{-1}X\right\rangle _{%
\mathcal{H}}$, and then relating that variance in the case of Wiener chaos
with the 4th cumulant (centered fourth moment) of $X$. However, this
strategy was superseded by the following, which has roots in \cite{BBNP}.

\item \textbf{The optimal fourth moment theorem}. For each integer $n$, let $%
X_{n}\in {\mathcal{H}}_{q}$. Assume $Var\left[ X_{n}\right] =1$ and $\left(
X_{n}\right) _{n}$ converges in distribution to a normal law. It is known
(original proof in \cite{NP}, known as the fourth moment theorem) that this
convergence is equivalent to $\lim_{n}\mathbf{E}\left[ X_{n}^{4}\right] =3$.
The following optimal estimate for $d_{TV}\left( X,N\right) $, known as the
optimal fourth moment theorem, was proved in \cite{NP2015}: with the
sequence $X$ as above, assuming convergence, there exist two constants $%
c,C>0 $ depending only on the law of $X$ but not on $n$, such that%
\begin{equation}
c\max \left\{ \mathbf{E}\left[ X_{n}^{4}\right] -3,\left\vert \mathbf{E}%
\left[ X_{n}^{3}\right] \right\vert \right\} \leqslant d_{TV}\left(
X_{n},N\right) \leqslant C\max \left\{ \mathbf{E}\left[ X_{n}^{4}\right]
-3,\left\vert \mathbf{E}\left[ X_{n}^{3}\right] \right\vert \right\} .
\label{optimal berry esseen}
\end{equation}
\end{itemize}
\end{itemize}


\section{Fluctuations of $\protect\rho (T)$ under $H_{a}$ : CLTs and rates
of convergence\label{FLUCTU}}

In this section, we study the detailed asymptotics of the law of the
empirical correlation coefficient $\rho (T)$ between our two OU paths $X_{1}$%
, $X_{2}$, under the alternative hypothesis of a non-zero true correlation
between them, when the time horizon $T\rightarrow +\infty $. As mentioned,
we interpret quantitatively the fact that $X_{1}$ and $X_{2}$ are correlated
by letting the correlation coefficient $r$ between the driving noises $W_{1}$
and $W_{2}$ be a fixed non-zero value: $r\in \lbrack -1,1]\backslash \{0\}$,
which is our alternative hypothesis $H_{a}$, while the null hypothesis $%
H_{0} $ is $r=0$. These hypothses are identical to assuming that $X_{1}$ and 
$X_{2} $ have a fixed non-zero correlation, and have a zero correlation,
respectively. Since all these processes are Gaussian, $H_{0}$ is equivalent
to independence of the pairs $\left( X_{1},X_{2}\right) $ or $\left(
W_{1},W_{2}\right) $.

To facilitate the mathematical analysis quantitatively, we introduce a
Brownian motion $W_{0}$ defined on the same probability space $(\Omega ,%
\mathcal{F},\mathbb{P})$ as $W_{1}$ and assumed to be independent of $W_{1}$%
. We then realize the Brownian motion $W_{2}$ on this probability space from 
$W_{1}$ and $W_{0}$ via the following elementary construction: for any $%
t\geq 0$, 
\begin{equation}
W_{2}(t):=rW_{1}(t)+\sqrt{1-r^{2}}W_{0}(t)  \label{W2}
\end{equation}%
The two OU paths $X_{1},X_{2}$ are still given via their SDEs (\ref{OU}).
Recall that we defined their empirical correlation coefficient, in (\ref{rho}%
), as 
\begin{equation}
\rho (T):=\frac{Y_{12}(T)}{\sqrt{Y_{11}(T)}\sqrt{Y_{22}(T)}},  \label{rho'}
\end{equation}%
where the random variables $Y_{ij}(T)$, $i,j=1,2$ are given as in (\ref{Y})
by 
\begin{equation}
Y_{ij}(T):=\int_{0}^{T}X_{i}(u)X_{j}(u)du-T\bar{X}_{i}\left( T\right) \bar{%
X_{j}}\left( T\right) ,\text{ \ }\text{ \ }\bar{X}_{i}\left( T\right) :=%
\frac{1}{T}\int_{0}^{T}X_{i}(u)du,  \label{Y'}
\end{equation}

\subsection{Gaussian fluctuations of the numerator $Y_{12}(T)$}

The numerator $Y_{12}(T)$ is defined as follows : 
\begin{equation*}
Y_{12}(T) = \int_{0}^{T} X_1(u) X_2(u) du - T \bar{X}_1(T) \bar{X_2}(T)
\end{equation*}
From the construction (\ref{W2}), we can write for any $0 \leqslant u
\leqslant T $ 
\begin{align*}
X_1(u) X_2(u) & = \left[r \int_{0}^{u}e^{- \theta (u-t)} dW_1(t) + \sqrt{%
1-r^2} \int_{0}^{u}e^{- \theta (u-t)} dW_0(t) \right] \times
\int_{0}^{u}e^{- \theta (u-t)} dW_1(t) \\
& = r I^{W_1}_1\left(f_u\right)^2 + \sqrt{1-r^2} I^{W_0}_1(f_u)I^{W_1}_1(f_u)
\\
& = r \left[ I^{W_1}_2(f^{\otimes 2}_u) + \| f_u\|^2_{\mathcal{H}}\right]+ 
\sqrt{1-r^2} I^{W_0}_1(f_u)I^{W_1}_1(f_u).
\end{align*}
where $f_u(.) := e^{-\theta(u-.)} \mathbf{1}_{[0,u]}(.)$, $\mathcal{H} :=
L^{2}([0,T]).$ On the other hand, using a rotational trick, and the
linearity of Wiener integrals, we can write 
\begin{align*}
I^{W_0}_1(f_u)I^{W_1}_1(f_u) & = \frac{1}{2} \left[ \left(\frac{%
I^{W_0}_1(f_u) + I^{W_1}_1(f_u)}{\sqrt{2}}\right)^2 - \left(\frac{%
I^{W_0}_1(f_u) - I^{W_1}_1(f_u)}{\sqrt{2}}\right)^2 \right] \\
& = \frac{1}{2} \left[(I^{\mathcal{U}_1}_1(f_u))^2 - (I^{\mathcal{U}%
_0}_1(f_u))^2 \right]
\end{align*}
where $\mathcal{U}_0 := \frac{W_1 - W_0}{\sqrt{2}}$ , $\mathcal{U}_1 := 
\frac{W_1 + W_0}{\sqrt{2}}$. Therefore, using the product formula (\ref%
{product-formula}) 
\begin{align*}
& \sqrt{1-r^2} \int_{0}^{T} I^{W_0}_1(f_u)I^{W_1}_1(f_u) du \\
& = \frac{\sqrt{1-r^2}}{{2}} \int_{0}^{T} I^{\mathcal{U}_1}_1(f_u)^2 du - 
\frac{\sqrt{1-r^2}}{{2}}\int_{0}^{T} I^{\mathcal{U}_0}_1(f_u)^2 du \\
& = \frac{\sqrt{1-r^2}}{{2}} \int_{0}^{T} I^{\mathcal{U}_1}_2(f^{\otimes
2}_u) du + \frac{\sqrt{1-r^2}}{{2}} \int_{0}^{T} \| f_u\|^2_{\mathcal{H}} du
- \frac{\sqrt{1-r^2}}{{2}} \int_{0}^{T} I^{\mathcal{U}_0}_2(f^{\otimes 2}_u)
du- \frac{\sqrt{1-r^2}}{{2}} \int_{0}^{T} \| f_u\|^2_{\mathcal{H}} du. \\
& = \frac{\sqrt{1-r^2}}{{2}} \int_{0}^{T} I^{\mathcal{U}_1}_2(f^{\otimes
2}_u) du - \frac{\sqrt{1-r^2}}{{2}} \int_{0}^{T} I^{\mathcal{U}%
_0}_2(f^{\otimes 2}_u) du.
\end{align*}
Moreover, we can write 
\begin{align*}
r\int_{0}^{T} \left[ I^{W_1}_2(f^{\otimes 2}_u) + \| f_u\|^2_{\mathcal{H}}%
\right] du & = r \int_{0}^{T} I_2^{\frac{\sqrt{2}}{2}(\mathcal{U}_1 + 
\mathcal{U}_0)}(f^{\otimes 2}_u) du + r \int_{0}^{T} \| f_u\|^2_{\mathcal{H}%
} du. \\
& = \frac{r \sqrt{2}}{2} \int_{0}^{T} I^{\mathcal{U}_1}_2(f^{\otimes 2}_u)
du + \frac{r \sqrt{2}}{2} \int_{0}^{T} I^{\mathcal{U}_0}_2(f^{\otimes 2}_u)
du + r \int_{0}^{T} \| f_u\|^2_{\mathcal{H}} du.
\end{align*}
Therefore, we can write 
\begin{align*}
\int_{0}^{T} X_1(u) X_2(u) du & = \left[ \frac{r\sqrt{2}}{2} + \frac{\sqrt{%
1-r^2}}{2} \right] \int_{0}^{T} I^{\mathcal{U}_1}_2(f^{\otimes 2}_u) du + %
\left[ \frac{r \sqrt{2}}{2} - \frac{\sqrt{1-r^2}}{2} \right] \int_{0}^{T} I^{%
\mathcal{U}_0}_2(f^{\otimes 2}_u) du + r \int_{0}^{T} \| f_u\|^2_{\mathcal{H}%
} du.
\end{align*}
It follows that : 
\begin{align*}
\frac{1}{\sqrt{T}} \int_{0}^{T} X_1(u) X_2(u) du & := A_r(T)+ \frac{r}{\sqrt{%
T}} \int_{0}^{T} \| f_u\|^2_{\mathcal{H}} du = A_r(T) + \frac{r \sqrt{T}}{2
\theta} - \frac{r }{4 \theta^2 \sqrt{T}} (1-e^{-2 \theta T}).
\end{align*}
We therefore obtain the following expression for $\frac{Y_{12}(T)}{\sqrt{T}}%
. $ 
\begin{equation}  \label{decomp-num1}
\frac{Y_{12}(T)}{\sqrt{T}} = A_r(T) + \frac{r \sqrt{T}}{2 \theta} + O(\frac{1%
}{\sqrt{T}}) - \sqrt{T} \bar{X}_{1}(T) \bar{X}_{2}(T).
\end{equation}
The following theorem gives the Gaussian fluctuations of the numerator term
along with its speed of convergence for the Wasserstein distance.

\begin{theorem}
\label{CLT-num} There exists a constant $C(\theta,r)$ depending on $\theta$
and $r$ such that 
\begin{equation*}
d_{W}\left( \frac{1}{\sigma_{r,\theta}}\left(\frac{Y_{12}(T)}{\sqrt{T}} - 
\frac{r \sqrt{T}}{2 \theta}\right), \mathcal{N}\left(0, 1\right)\right)
\leqslant \frac{C(\theta,r)}{\sqrt{T}}
\end{equation*}
where $\sigma_{r,\theta} := \left(\frac{1}{2 \theta^3} \left(\frac{1}{2}+%
\frac{r^2}{2}\right)\right)^{1/2}.$ In particular, 
\begin{equation*}
\left(\frac{Y_{12}(T)}{\sqrt{T}} - \frac{r \sqrt{T}}{2 \theta} \right) 
\overset{\mathrm{\mathcal{L}}}{ \longrightarrow } \mathcal{N}\left(0, \frac{1%
}{2 \theta^3} \left(\frac{1}{2}+\frac{r^2}{2}\right)\right) \text{ \ } \text{%
as} \text{ \ } T \rightarrow +\infty.
\end{equation*}
\end{theorem}

\begin{proof}
We will first prove a CLT for the second- Wiener chaos term $A_r(T)$, indeed
we can write 
\begin{equation}  \label{ArT}
A_r(T) : = A_{r,1}(T) + A_{r,2}(T)
\end{equation}
We claim that as $T \rightarrow +\infty$ 
\begin{eqnarray*}
\left\{ 
\begin{array}{ll}
A_{r,1}(T) & := \frac{c_{1}(r)}{\sqrt{T}} \int_{0}^{T} I^{\mathcal{U}%
_1}_2(f^{\otimes 2}_u) du\overset{\mathrm{\mathcal{L}}}{ \longrightarrow } 
\mathcal{N}\left(0, \frac{c_{1}(r)^2}{2 \theta^3} \right) \\ 
A_{r,2}(T) & := \frac{c_2(r)}{\sqrt{T}} \int_{0}^{T} I^{\mathcal{U}%
_0}_2(f^{\otimes 2}_u) du\overset{\mathrm{\mathcal{L}}}{ \longrightarrow } 
\mathcal{N}\left(0, \frac{ c_2(r)^2}{2 \theta^3}\right)%
\end{array}
\right.
\end{eqnarray*}
where : 
\begin{equation}  \label{cts}
c_1(r) = \frac{r\sqrt{2}}{2} + \frac{\sqrt{1-r^2}}{2}, \ \ \ c_2(r) = \frac{%
r \sqrt{2}}{2} - \frac{\sqrt{1-r^2}}{2}.
\end{equation}
We suggest first to compute the third and fourth cumulant of the term $%
A_r(T) $ in order to use the Optimal fourth moment theorem (\ref{optimal
berry esseen}). Since $A_{r,1}(T)$ and $A_{r,2}(T)$ are centered and using
the independence of $\mathcal{U}_1$ and $\mathcal{U}_0$, we can write 
\begin{align*}
k_3(A_r(T)) & = {E}[A_r(T)^3] \\
& = E[(A_{r,1}(T) +A_{r,2}(T))^3] \\
& = E[A_{r,1}(T)^3] + 3 E[A_{r,1}(T)^2]E[A_{r,2}(T)] + 3
E[A_{r,1}(T)]E[A_{r,2}(T)^3] + E[A_{r,2}(T)^3] \\
& = E[A_{r,1}(T)^3] +E[A_{r,2}(T)^3] \\
& = k_3(A_{r,1}(T)) + k_3(A_{r,2}(T)).
\end{align*}
For the fourth cumulant, we have 
\begin{align*}
E[A_{r}(T)^4] & = E[(A_{r,1}(T) +A_{r,2}(T))^4] \\
& = E[A_{r,1}(T)^4] + 4 E[A_{r,1}(T)^3A_{r,2}(T)] + 6 E[A_{r,1}(T)^2
A_{r,2}(T)^2] + 4 E[A_{r,1}(T)A_{r,2}(T)^3]+ E[A_{r,2}(T)^4] \\
& = E[A_{r,1}(T)^4] + 6 E[A_{r,1}(T)^2] E[A_{r,2}(T)^2] + E[A_{r,2}(T)^4].
\end{align*}
Therefore 
\begin{align*}
k_4(A_r(T)) & = E[A_r(T)^4] - 3 E[A_r(T)^2]^2 \\
& = (E[A_{r,1}(T)^4] -3 E[A_{r,1}(T)^2]^2) + (E[A_{r,2}(T)^4] - 3
E[A_{r,2}(T)^2]^2) \\
& = k_4(A_{r,1}(T)) + k_4(A_{r,2}(T)).
\end{align*}

\begin{proposition}
\label{cumulants} There exists constants $c_1(\theta,r)$, $c_2(\theta,r)$
defined as follows 
\begin{equation}  \label{constantes}
c_i(\theta,r) = \max\left( \frac{16}{9} \frac{1}{\theta^5} \left| c_i(r)^3
\right|, \frac{81}{ 8 \theta^7} c_i(r)^4\right), i =1,2.
\end{equation}
where the constants $c_1(r)$ and $c_2(r)$ are defined in (\ref{cts}). Then,
we have for $i=1,2$ : 
\begin{equation*}
\max\left\{k_3(A_{r,i}(T)), k_4(A_{r,i}(T)) \right\} \leqslant \frac{%
c_i(\theta,r)}{\sqrt{T}}.
\end{equation*}
\end{proposition}

\begin{proof}
The terms $A_{r,1}(T)$ and $A_{r,2}(T)$ can be treated similarly, we will do
the computations just for $A_{r,1}(T)$. We can write 
\begin{equation*}
A_{r,1}(T) = I^{\mathcal{U}_1}_2(g_{r,T}),
\end{equation*}
with 
\begin{equation}  \label{kernel-g}
g_{r,T} := \frac{c_1(r)}{\sqrt{T}} \int_{0}^{T} f^{\otimes 2}_t dt.
\end{equation}
Therefore, using the definition of the third cumulant and since ${E}%
[X_1(r)X_1(s)]=\frac{ e^{-\theta (r+s)}}{2\theta }[e^{2\theta (r\wedge
s)}-1]\leqslant \frac{1}{ 2\theta }e^{-\theta |r-s|}:=\delta(r-s)$, we get 
\begin{align*}
k_3(A_{r,1}(T))& = 8\left<g_{r,T},g_{r,T}\otimes_1 g_{r,T}\right>_{\mathcal{H%
}^{\otimes 2}} \\
& = 8 \int_{0}^{T}\int_{0}^{T} g_{r,T}(x,y) (g_{r,T}\otimes_1 g_{r,T})(x,y)
dx dy \\
& = 8 \int_{0}^{T}\int_{0}^{T}\int_{0}^{T} g_{r,T}(x,y) g_{r,T}(z,y)
g_{r,T}(x,z) dx dy dz \\
& = \frac{8 \times c_1(r)^3 }{T^{3/2}}\int_{[0,T]^6} f_{u}(x)
f_{u}(y)f_{v}(x)f_{v}(z)f_{r}(y)f_{r}(z) du dv dr dx dy dz \\
& = \frac{8 \times c_1(r)^3 }{T^{3/2}} \int_{[0,T]^3} \langle
f_{u},f_{v}\rangle\left<f_{u},f_{r}\right>\left<f_{v},f_{r}\right> du dv dr
\\
& = \frac{8 \times c_1(r)^3 }{T^{3/2}} \int_{[0,T]^3} E[X_1(u)X_1(v)]
E[X_1(u)X_1(v)] E[X_1(u)X_1(v)] du dv dr.
\end{align*}
It follows that : 
\begin{align*}
\left| k_3(A_{r,1}(T)) \right| & \leqslant \frac{8}{T^{3/2}} \left| c_1(r)^3
\right| \left|\int_{[0,T]^3}\delta(u-v)\delta(v-r)\delta(u-r)dudvdr \right|
: = \left|k_3(F_T) \right|, & &
\end{align*}
where $F_T := I^{\mathcal{U}_{1}}_2 \left(c_1(r) \delta(t-s) \mathbf{1}%
(t,s)_{[0,T]^2}\right)$. We proved in Proposition \ref{equiv-cumulants} in
the Appendix that : 
\begin{equation}  \label{cum}
\forall p\geq 3, \ \ k_p\left({F}_T\right) \underset{+\infty}{\sim} \frac{%
c_1(r)^{p} \langle \delta^{*(p-1)}, \delta\rangle_{\mathcal{L}^{2}(\mathbb{R}%
)} 2^{p-1} (p-1)! }{T^{p/2-1}}.
\end{equation}
where $\delta^{*(p)}$ denotes the convolution of $\delta$ p times defined as 
$\delta^{*(p)} = \delta^{*(p-1)}* \delta$, $p \geq 2$, $\delta^{*(1)} =
\delta$ where $*$ denotes the convolution between two functions $(f*g)(x) =
\int_{\mathbb{R}} f(x-y)g(y) dy$. \newline
It follows that there exists $T_0 > 0$, such that for all $T \geq T_0$, 
\begin{align*}
\left| k_3(A_{r,1}(T)) \right| & \leqslant \left|c_1(r)^{3}\right| \times 
\frac{8 |\langle \delta^{*(2)}, \delta\rangle_{\mathcal{L}^{2}(\mathbb{R})}|%
}{\sqrt{T}} \\
& \leqslant \frac{16}{9 \theta^5} \left|c_1(r)^{3}\right| \frac{1}{\sqrt{T}}.
\end{align*}
In fact, for the last inequality we will make use of Young's inequality that
we recall here: if $p,q,s\geq1$ are such that $\frac{1}{p}+\frac{1}{q}=\frac{%
1}{s}+1$, and $f\in L^p(\mathbb{R})$, $g\in L^q(\mathbb{R})$, then 
\begin{eqnarray}
\|f*g\|_{L^s(\mathbb{R})}\leqslant\|f\|_{L^p(\mathbb{R})}\|g\|_{L^q(\mathbb{R%
})}.  \label{young-inequality}
\end{eqnarray}
Therefore, using first Hölder inequality and then Young's inequality (\ref%
{young-inequality}), with $p=q=\frac32$ and $s=3$, we get : 
\begin{align*}
|\langle \delta^{*(2)}, \delta\rangle_{\mathcal{L}^{2}(\mathbb{R})} &
\leqslant \| \delta * \delta \|_{L^3(\mathbb{R})} \| \delta \|_{L^{3/2}(%
\mathbb{R})} \\
&\leqslant \| \delta \|^3_{L^{3/2}(\mathbb{R})} = \left( \int_{\mathbb{R}}
\left(\frac{1}{2 \theta} e^{- \theta |u|} \right)^{3/2} du\right)^2 := \frac{%
2}{9} \frac{1}{\theta^5}.
\end{align*}
For the fourth cumulant, we have : 
\begin{eqnarray*}
\left|k_{4}(A_{r,1}(T))\right| &=&16\left(\|g_{r,T}\otimes_1 g_{r,T}\|_{%
\mathcal{H}^{\otimes 2}}^2+2\|g_{r,T}\widetilde{\otimes_1} g_{r,T}\|_{%
\mathcal{H}^{\otimes 2}}^2\right) \\
&\leqslant&48\|g_{r,T}\otimes_1 g_{r,T}\|_{\mathcal{H}^{\otimes 2}}^2 \\
& =& 48 \int_{[0,T]^2} \left(g_{r,T}\otimes_1 g_{r,T}\right)^{2}(x,y) dx dy
\\
& =& 48 \int_{[0,T]^4} g_{r,T}(x,z) g_{r,T}(x,t) g_{r,T}(z,y) g_{r,T}(t,y)
dt dz dx dy \\
& =& \frac{48 \times c_1(r)^{4} }{T^2} \int_{[0,T]^8} f_u(x) f_u(z) f_v(x)
f_v(t) f_r(z) f_r(y) f_s(t) f_s(y) du dv dt dxdy dz dt \\
& = & \frac{48 \times c_1(r)^{4}}{T^2} \int_{[0,T]^4} E[X_1(u)X_1(v)]
E[X_1(u) X_1(r)] E[X_1(v)X_1(s)] E[X_1(r)X_1(s)]du dv dr ds.
\end{eqnarray*}
It follows that : 
\begin{align*}
\left| k_4(A_{r,1}(T)) \right|& \leqslant 48 \times c_1(r)^{4} \left|
\int_{[0,T]^4}\delta(u-v)\delta(v-r)\delta(r-s)\delta(s-u)dudvdrds \right| :
= \left|k_4(F_T) \right|,
\end{align*}
Using, the equivalent (\ref{cum}), it follows that that there exists $T_0 >
0 $, such that for all $T \geq T_0$, 
\begin{align*}
\left| k_4(A_{r,1}(T)) \right| & \leqslant c_1(r)^{4} \times \frac{48
|\langle \delta^{*(3)}, \delta\rangle_{\mathcal{L}^{2}(\mathbb{R})}|}{{T}}
\leqslant \frac{81}{8 \theta^7} \frac{c_1(r)^{4}}{{T}}.
\end{align*}
In fact, using first Hölder inequality and then Young's inequality (\ref%
{young-inequality}), we get : 
\begin{align*}
|\langle \delta^{*(3)}, \delta\rangle_{\mathcal{L}^{2}(\mathbb{R})}| &
\leqslant \| \delta\|_{L^4/3(\mathbb{R})} \times \| \delta^{*(3)} \|_{L^{4}(%
\mathbb{R})} := \| \delta\|_{L^4/3(\mathbb{R})} \times \| \delta^{*(2)} *
\delta \|_{L^{4}(\mathbb{R})} \\
& \leqslant \| \delta\|^2_{L^{4/3}(\mathbb{R})} \times \| \delta^{*(2)}
\|_{L^{2}(\mathbb{R})} \\
& \leqslant \| \delta\|^4_{L^{4/3}(\mathbb{R})} := \left( \int_{\mathbb{R}}
\left(\frac{1}{2 \theta} e^{- \theta |u|} \right)^{4/3} du\right)^3 = \frac{%
27}{128} \frac{1}{\theta^7}.
\end{align*}
\end{proof}
\end{proof}

\begin{proposition}
\label{estim-var} There exists a constant $C(\theta,r)>0$ such that 
\begin{equation*}
\left|E[A^2_r(T)] - \sigma^2_{r,\theta}\right| \leqslant \frac{C(\theta,r)}{T%
}.
\end{equation*}
In particular, when $T \rightarrow + \infty$, we have 
\begin{equation*}
E[A^2_r(T)] \rightarrow \sigma^2_{r,\theta}.
\end{equation*}
\end{proposition}

\begin{proof}
By the decomposition (\ref{ArT}), and the independence between $A_{r,1}(T)$
and $A_{r,2}(T)$, we have 
\begin{align*}
\left|E[A^2_r(T)] - \sigma^2_{r,\theta}\right| \leqslant
\left|E[A^2_{r,1}(T)] - \frac{1}{2 \theta^3} \left(\frac{r \sqrt{2}}{2} + 
\frac{\sqrt{1-r^2}}{2}\right)^2 \right| + \left|E[A^2_{r,2}(T)] - \frac{1}{2
\theta^3} \left(\frac{r \sqrt{2}}{2} - \frac{\sqrt{1-r^2}}{2}\right)^2
\right|
\end{align*}
Both left hand sided terms can be treated similarly, in fact it suffices to
show that for $i=0,1$, 
\begin{equation*}
\left|E \left[\left(\frac{1}{\sqrt{T}} \int_{0}^{T} I^{U_i}_2(f^{\otimes}_t)
dt\right)^2 \right] - \frac{1}{2 \theta^3}\right| = O(\frac{1}{T}).
\end{equation*}
By the isometry property (\ref{isometry}), we have 
\begin{align*}
& E \left[I^{U_i}_2\left(\frac{1}{\sqrt{T}} \int_{0}^{T} f^{\otimes 2}_t
dt\right)^2 \right] = 2 \| \frac{1}{\sqrt{T}} \int_{0}^{T} f^{\otimes 2}_t
dt\|_{L^{2}[0,T]^2} \\
& = \frac{2}{T} \int_{0}^{T} \int_{0}^{T} \langle f^{\otimes 2}_t,
f^{\otimes 2}_s \rangle dt ds \\
& = \frac{2}{T} \int_{0}^{T} \int_{0}^{T} \left( \langle f_t, f_s \rangle
\right)^2 dt ds \\
& = \frac{2}{T} \int_{0}^{T} \int_{0}^{T} \left(\int_{0}^{t \wedge s} e^{-
\theta(t-u)} e^{- \theta(s-u)} du\right)^2 dt ds \\
& = \frac{1}{\theta^2} \frac{1}{T} \int_{0}^{T} \int_{0}^{t} e^{- 2 \theta t
} e^{- 2 \theta s } ( e^{ 2 \theta s }-1)^2 dt ds \\
& = \frac{1}{2 \theta^3} \frac{1}{T} \left(T - \int_{0}^{T} e^{- 2 \theta t}
dt\right) - \frac{2}{\theta^2 T} \int_{0}^{T} t e^{- 2 \theta t} dt + \frac{1%
}{2 \theta^3} \frac{1}{T} \int_{0}^{T}e^{-2 \theta t}dt - \frac{1}{2 \theta^3%
} \frac{1}{T} \int_{0}^{T}e^{-4 \theta t } dt \\
& = \frac{1}{2 \theta^3} - \frac{3}{4} \frac{1}{\theta^4 T} (1- e^{-2 \theta
T}) + \frac{3}{\theta^3} e^{-2 \theta T} + \frac{1}{4 \theta^3} \frac{1}{T}
+ \frac{1}{8 \theta^4} \frac{1}{T} (e^{-4 \theta T} -1).
\end{align*}
The desired result follows.
\end{proof}

\begin{proposition}
Consider $A_r(T)$ defined previously in (\ref{ArT}), then there exists a
constant $C$ depending only on $\theta $ and $r$ but not on $T$, such that : 
\begin{equation}  \label{asser-1}
d_W\left(\frac{A_r(T)}{E[A^2_r(T)]^{1/2}}, \mathcal{N}(0,1)\right) \leqslant 
\frac{C}{E[A^2_r(T)]^2 \wedge E[A^2_r(T)]^{3/2}} \times \frac{1}{\sqrt{T}}.
\end{equation}
Moreover, there exists a constant $C$ depending on $\theta$ and $r$ such
that 
\begin{equation}  \label{asser-2}
d_W\left( \frac{1}{\sigma_{r,\theta}} A_r(T), \mathcal{N}(0,1) \right)
\leqslant \frac{C}{\sqrt{T}}.
\end{equation}
with $\sigma_{r,\theta} := \left(\frac{1}{2 \theta^3} \left(\frac{1}{2}+%
\frac{r^2}{2}\right)\right)^{1/2}.$
\end{proposition}

\begin{proof}
First observe that the term $A_r(T)$ defined in (\ref{ArT}) is a second
Wiener chaos term, with respect to a two sided Brownian motion $(W(t))_{t\in 
\mathbb{R }}$ that we can construct from $(\mathcal{U}_0(t))_{t\geq 0}$ and $%
(\mathcal{U}_1(t))_{t\geq 0}$ as follows : 
\begin{equation*}
W(t) := \mathcal{U}_1(t) \mathbf{1}_{\{t\geq 0\}}+ \mathcal{U}_0(-t)\mathbf{1%
}_{\{t<0\}},\text{ \ }t\in \mathbb{R.}
\end{equation*}
It is therefore easy to check that the following equality holds in law. 
\begin{align*}
A_r(T) & = \frac{1}{\sqrt{T}} \left[ \frac{r\sqrt{2}}{2} + \frac{\sqrt{1-r^2}%
}{2} \right] \int_{0}^{T} I^{\mathcal{U}_1}_2(f^{\otimes 2}_u) du + \frac{1}{%
\sqrt{T}} \left[ \frac{r \sqrt{2}}{2} - \frac{\sqrt{1-r^2}}{2} \right]
\int_{0}^{T} I^{\mathcal{U}_0}_2(f^{\otimes 2}_u) du \\
& \overset{\mathrm{"law"}}{ = } I^{W}_2 \left( \frac{1}{\sqrt{T}}
\int_{0}^{T} \left(\left[ \frac{r\sqrt{2}}{2} + \frac{\sqrt{1-r^2}}{2} %
\right] f^{\otimes 2}_u + \left[ \frac{r \sqrt{2}}{2} - \frac{\sqrt{1-r^2}}{2%
} \right] \bar{\bar{f}}^{\otimes 2}_u \right)du \right)
\end{align*}
where 
\begin{equation*}
\bar{\bar{f}}(x) = -f(-x) \mathbf{1}_{\{x < 0\}}.
\end{equation*}
Therefore, it is possible to apply the Optimal fourth moment theorem (\ref%
{optimal berry esseen}) to the term $\frac{A_r(T)}{E[A^{2}_r(T)]^{1/2}}$, we
get : 
\begin{align}  \label{opt}
d_W\left(\frac{A_r(T)}{E[A^2_r(T)]^{1/2}}, \mathcal{N}(0,1)\right) & \asymp
\max\left\{ k_3\left( \frac{A_r(T)}{E[A^{2}_r(T)]^{1/2}}\right), k_4\left( 
\frac{A_r(T)}{E[A^{2}_r(T)]^{1/2}}\right)\right\}.
\end{align}
Hence, using Proposition \ref{cumulants}, we get the following estimate: 
\begin{align}
d_W\left(\frac{A_r(T)}{E[A^2_r(T)]^{1/2}}, \mathcal{N}(0,1)\right) &
\leqslant \frac{C \times \left(c_1(\theta,r)+c_2(\theta,r) \right)}{%
E[A^2_r(T)]^2 \wedge E[A^2_r(T)]^{3/2}} \times \frac{1}{\sqrt{T}}.
\end{align}
where $C$ is a constant coming from (\ref{opt}) and $c_1(\theta,r)$, $%
c_2(\theta,r)$ are defined in (\ref{constantes}). For (\ref{asser-2}), we
will need the following proposition.

\begin{proposition}
\label{estim-dw} Let $N \sim \mathcal{N}(0,1)$, and $\sigma >0$, then 
\begin{equation*}
d_W(\sigma N, N) \leqslant \frac{\sqrt{2}}{\sqrt{\pi}} \left|1-\sigma^2
\right|
\end{equation*}
For $\mu \in \mathbb{R}$, $F \in L^2(\Omega)$, $Y \in L^1(\Omega)$, we have 
\begin{equation*}
d_W(\sigma F + \mu+ Y,N) \leqslant |\mu| + E[|Y|] + \sigma d_W(F,N) + \frac{%
\sqrt{2}}{\sqrt{\pi}} \left|1-\sigma^2 \right|.
\end{equation*}
\end{proposition}

\begin{proof}
Let $N \sim \mathcal{N}(0,1)$, using the Stein's caracterisation of $d_W$,
we get the following estimate : 
\begin{align*}
d_W(\sigma N, N) & = \sup\limits_{h \in lip(1)} \left|E[h(\sigma N)]-
E[h(N)] \right| \\
& \leqslant \sup\limits_{f \in \mathcal{F}_W} \left| E[f^{\prime}(\sigma N)
- \sigma N f(\sigma N)] \right|
\end{align*}
where $\mathcal{F}_W := \left\{ f : \mathbb{R} \rightarrow \mathbb{R} \in 
\mathcal{C}^1 : \|f^{\prime} \|_{\infty} \leqslant \sqrt{2/\pi} \right\}.$
By an integration by parts, we have 
\begin{align*}
E[N f(\sigma N)] & = \frac{1}{\sqrt{2 \pi}} \int_{\mathbb{R}} x f(\sigma x)
e^{- \frac{x^2}{2}} dx \\
& = \frac{\sigma}{\sqrt{2 \pi}} \int_{\mathbb{R}} f^{\prime}(\sigma x) e^{- 
\frac{x^2}{2}} dx \\
& = \sigma E[f^{\prime}(\sigma N)].
\end{align*}
It follows that 
\begin{equation*}
d_W(\sigma N, N) \leqslant \frac{\sqrt{2}}{\sqrt{\pi}} \left|1-\sigma^2
\right|
\end{equation*}
On the other hand, we have for $F \in L^{2}(\Omega)$, $Y \in {L^1}(\Omega)$, 
$\mu \in \mathbb{R}$, $\sigma >0$. 
\begin{align*}
d_W(\sigma F + \mu+ Y,N) & = \sup \limits_{h \in lip(1)} \left|E[h(\sigma F
+\mu + Y)] - E[h(N)]\right| \\
& \leqslant \sup \limits_{h \in lip(1)} \left|E[h(\sigma F +\mu + Y)] -
E[h(\sigma F)] \right| + \sup \limits_{h \in lip(1)}\left|E[h(\sigma F)] -
E[h(N)]\right| \\
& \leqslant |\mu| + E[|Y|] + d_W(\sigma F, N).
\end{align*}
Using triangular inequality, we have $d_W(\sigma F, N) \leqslant d_W(\sigma
F, \sigma N) + d_W(\sigma N,N)$. Moreover, we can check that $d_W(\sigma F,
\sigma N) = \sigma d_W(F,N).$ Indeed using the definition of the Wasserstein
distance, we have 
\begin{align*}
d_W(\sigma F, \sigma N) & = \sup \limits_{h \in lip(1)} |E[h(\sigma F)] -
E[h(\sigma N)]| \\
& = \sup \limits_{h \in lip(\sigma)} |E[h(F)] - E[h(N)]|
\end{align*}
Similarly, 
\begin{align*}
\sigma d_W(F,N) & = \sigma \sup \limits_{h \in lip(1)} |E[h(F)] - E[h(N)]| \\
& = \sup \limits_{h \in lip(1)} |E[(\sigma h)(F)]- E[(\sigma h)(N)]| \\
& = \sup \limits_{h \in lip(\sigma)} |E[h(F)] -E[h(N)]|.
\end{align*}
The desired result follows using (\ref{asser-1}).
\end{proof}

It follows from Propositions (\ref{estim-dw}) and (\ref{estim-var}), that 
\begin{align*}
d_W\left(\frac{A_r(T)}{\sigma_{r,\theta}} , \mathcal{N}(0,1)\right) &
\leqslant \frac{E[A^2_r(T)]^{1/2}}{\sigma_{r,\theta}} d_W\left( \frac{A_r(T)%
}{E[A^2_r(T)]^{1/2}}, \mathcal{N}(0,1)\right) + \frac{\sqrt{2}}{\sqrt{\pi}}
\left|1- \frac{E[A^2_r(T)]}{\sigma^2_{r,\theta}} \right| \\
& \leqslant \frac{C}{\sigma_{r,\theta}} \frac{1}{E[A^2_{r,T}] \wedge
E[A^2_r(T)]^{3/2}} \times \frac{1}{\sqrt{T}} + \frac{\sqrt{2}}{\sqrt{\pi}} 
\frac{1}{T} \leqslant\frac{C(\theta,r)}{\sqrt{T}}.
\end{align*}
On the other hand, from decomposition (\ref{decomp-num}), we can write 
\begin{equation*}
\frac{1}{\sigma_{r,\theta}} \left(\frac{Y_{12}(T)}{\sqrt{T}} - \frac{r \sqrt{%
T}}{2 \theta}\right) = \frac{1}{\sigma_{r,\theta}} A_r(T) + \mu_{\theta}(T)
+ Y_{\theta}(T)
\end{equation*}
where $\mu_{\theta}(T) = O(\frac{1}{\sqrt{T}})$, $Y_{\theta}(T) := \sqrt{T} 
\bar{X}_1(T) \bar{X}_2(T)$, therefore, by Proposition \ref{estim-dw} and
estimate (\ref{asser-2}), we can write 
\begin{equation*}
d_W\left( \frac{1}{\sigma_{r,\theta}} \left(\frac{Y_{12}(T)}{\sqrt{T}} - 
\frac{r \sqrt{T}}{2 \theta}\right), \mathcal{N}(0,1) \right) \leqslant
|\mu_{\theta}(T)|+ E[|Y_{\theta}(T)|] + \frac{C(\theta,r)}{\sqrt{T}}.
\end{equation*}
For the term $E[|Y_{\theta}(T)|]$, we have $X_2(u) = r X_1(u) + \sqrt{1-r^2}
X_0(u)$, where $X_0$ is the Ornstein-Uhlenbeck driven by the Brownian motion 
$W_0$ considered in the beginning of this section. We can therefore write 
\begin{align*}
E[|Y_{\theta}(T)|] & \leqslant r E[\bar{X}^2_1(T)]+ \sqrt{1-r^2} E[|\bar{X}%
_1(T) \bar{X}_2(T)|] \\
& \leqslant r E[\bar{X}^2_1(T)] + \sqrt{1-r^2} E[\bar{X}^2_1(T)]^{1/2} E[%
\bar{X}^2_2(T)]^{1/2}
\end{align*}
On the other hand, 
\begin{align}
\mathbf{E}[\bar{X_{i}}^{2}(T)]& =\frac{1}{T^{2}}\int_{0}^{T}(%
\int_{0}^{T}f_{t}(u)dt)^{2}du  \notag \\
& =\frac{1}{T^{2}}\int_{0}^{T}e^{2\theta u}(\int_{u}^{T}e^{-\theta
t}dt)^{2}du  \notag \\
& =\frac{1}{T^{2}}\frac{1}{\theta ^{2}}\int_{0}^{T}(1-e^{-\theta
(T-u)})^{2}du\leqslant \frac{1}{\theta ^{2}}\frac{1}{T}.  \label{normbarX}
\end{align}
It follows that 
\begin{align}  \label{Ytheta}
E[|Y_{\theta}(T)|] & \leqslant \frac{1}{\sqrt{T}} \frac{1}{\theta^2} (r+ 
\sqrt{1-r^2}).
\end{align}
which finishes the proof of Theorem 1.
\end{proof}

\begin{proposition}
\label{norm-denom} Let $p \geq 1$, then there exists a constant depending
only on $\theta$ and $p$, such that 
\begin{equation*}
E \left[ \left|2 \theta \sqrt{\frac{Y_{11}(T)}{T} \times \frac{Y_{22}(T)}{T}}
-1 \right|^{p}\right]^{1/p} \leqslant \frac{c(p,\theta)}{\sqrt{T}}.
\end{equation*}
Moreover, as $T \rightarrow +\infty$, we have 
\begin{equation}  \label{a.sdenom}
\sqrt{\frac{Y_{11}(T)}{T} \times \frac{Y_{22}(T)}{T}} \overset{\mathrm{a.s.}}%
{ \longrightarrow } \frac{1}{2 \theta}.
\end{equation}
\end{proposition}

\begin{proof}
Using the fact that if $X \geq 0$ a.s. then for any $p \geq 1$, we have $E[|%
\sqrt{X}-1|^p]^{1/p} \leqslant E[|X-1|^{p}]^{1/p}$. Hence, using the
notation $\bar{Y}_{ii}(T):= 2 \theta\frac{Y_{ii}(T)}{T}$, for $i=1,2.$ By
Minkowski's and Holder's inequalities, we can write 
\begin{align*}
E\left[\left|\bar{Y}_{11}(T)\bar{Y}_{22}(T) -1\right|^{p} \right]^{1/p}&
\leqslant E\left[\left|\bar{Y}_{11}(T)(\bar{Y}_{22}(T)-1) \right|^{p} \right]%
^{1/p} + E \left[ \left| \bar{Y}_{11}(T) -1 \right|^{p}\right]^{1/p} \\
& \leqslant E \left[ \left|\bar{Y}_{11}(T)\right|^{2p}\right]^{1/2p} E\left[
\left|\bar{Y}_{22}(T)-1\right|^{2p}\right]^{1/2p} + E \left[ \left| \bar{Y}%
_{11}(T) -1 \right|^{p}\right]^{1/p}
\end{align*}
On the other hand, we can show that for any $p \geq 1$, there exists a
constant $c(p,\theta)$ such that 
\begin{equation*}
E \left[ \left| \bar{Y}_{11}(T) -1 \right|^{p}\right]^{1/p} \leqslant \frac{%
c(p,\theta)}{\sqrt{T}}.
\end{equation*}
where 
\begin{equation*}
c(p,\theta):= 3 \max\left\{ \frac{2(2p-1)}{\theta}, (p-1) \frac{\sqrt{2}}{%
\sqrt{\theta}}\left(3+\frac{7}{4 \theta}\right)^{1/2}, \frac{1}{2 \theta}%
\right\}
\end{equation*}
We have 
\begin{align*}
\bar{Y}_{11}(T)& =\frac{2\theta }{T}\int_{0}^{T}\left( X_{1}^{2}(u)-\mathbf{%
E }[X_{1}^{2}(u)]\right) du+\frac{2\theta }{T}\int_{0}^{T}\mathbf{E}
[X_{1}^{2}(u)]du-2\theta \bar{X}_{1}^{2}(T) \\
& =\frac{2\theta }{T}\int_{0}^{T}\left( (I_{1}^{W_{1}}(f_{u}))^{2}-\Vert
f_{u}\Vert _{L^{2}([0,T])}^{2}\right) du+\frac{1}{T}\int_{0}^{T}\mathbf{E}
[X_{1}^{2}(u)]du-\bar{X}_{1}^{2}(T) \\
& =2\theta I_{2}^{\mathcal{U}_{0}}(k_{T})+\frac{2\theta }{T}\int_{0}^{T}%
\mathbf{E} [X_{1}^{2}(u)]du-2\theta \bar{X}_{1}^{2}(T).
\end{align*}
where 
\begin{align*}
k_{T}(x,y)& :=\frac{1}{T}\int_{0}^{T}f_{u}^{\otimes 2}(x,y)du \\
& =\frac{1}{T}\int_{0}^{T}e^{-\theta (u-x)}e^{-\theta (u-y)}\mathbf{1}
_{[0,u]}(x)\mathbf{1}_{[0,u]}(y)du \\
& =\frac{1}{T}\frac{1}{2\theta }e^{\theta x}e^{\theta y}\left( e^{-2\theta
(x\vee y)}-e^{-2\theta T}\right) \mathbf{1}_{[0,T]}(x)\mathbf{1}_{[0,T]}(y)
\end{align*}
and 
\begin{align}  \label{determ-term}
\frac{2\theta }{T}\int_{0}^{T}\mathbf{E}[X_{1}^{2}(u)]du &= \frac{2\theta }{T%
}\int_{0}^{T}\Vert f_{u}\Vert _{L^{2}([0,T]}^{2}du  \notag \\
& =\frac{2\theta }{T}\int_{0}^{T}\int_{0}^{T}f_{u}^{2}(t)dtdu  \notag \\
& =\frac{2\theta }{T}\int_{0}^{T}\int_{0}^{u}e^{-2\theta (u-t)}dtdu  \notag
\\
& =\frac{1}{T}\int_{0}^{T}(1-e^{-2\theta u})du  \notag \\
& =1-\frac{1}{2\theta T}\left( 1-e^{-2\theta T}\right)
\end{align}
The following inequality holds for any $p \geq 1$, 
\begin{align}  \label{estim-barY}
E \left[ \left| \bar{Y}_{11}(T) -1 \right|^{p}\right]^{1/p} & \leqslant 2
\theta E[|I_{2}^{\mathcal{U}_{0}}(k_T)|^{p}]^{1/p} + \left| \frac{2\theta }{T%
}\int_{0}^{T}\mathbf{E}[X_{1}^{2}(u)]du-1\right| +2\theta E[|\bar{X}%
^{2p}_1(T)|]^{1/p} \\
& \leqslant (p-1) E[|I_{2}^{\mathcal{U}_{0}}(k_T)|^{2}]^{1/2} + \left| \frac{%
2\theta }{T}\int_{0}^{T}\mathbf{E}[X_{1}^{2}(u)]du-1\right| + 2 \theta
(2p-1) E[|\bar{X}^{2}_1(T)|]
\end{align}
where we used the hypercontractivity property for multiple Wiener integrals.
On the other hand, 
\begin{align}  \label{normI_2}
\mathbf{E}\left[ I_{2}^{W_{1}}(k_{T})^{2}\right] & =2\Vert k_{T}\Vert
_{L^{2}([0,T]^{2})}^{2}  \notag \\
& =\frac{1}{T^{2}}\frac{1}{2\theta ^{2}}\int_{[0,T]^{2}}e^{2\theta
x}e^{2\theta y}\left( e^{2\theta (x\vee y)}-e^{-2\theta T}\right) ^{2}dxdy 
\notag \\
& =\frac{1}{T^{2}}\frac{1}{2\theta ^{3}}\left( \frac{1}{4\theta }
(1-e^{-4\theta T})+\frac{1}{\theta }(e^{-2\theta T}-1)+T(1+2e^{-2\theta T})- 
\frac{1}{2\theta }(1-e^{-4\theta T})\right)  \notag \\
& \leqslant \frac{1}{2\theta ^{3}}\left( 3+\frac{7}{4\theta }\right) \frac{1 
}{T}.
\end{align}
The desired result follows from inequalities (\ref{estim-barY}), (\ref%
{normI_2}), (\ref{normbarX}) and (\ref{determ-term}). 
Notice that the denominator term of $\rho(T)$ also satisfies the fact that
as $T \rightarrow +\infty$, 
\begin{equation}  \label{a.sdenom}
\sqrt{\frac{Y_{11}(T)}{T} \times \frac{Y_{22}(T)}{T}} \overset{\mathrm{a.s.}}%
{ \longrightarrow } \frac{1}{2 \theta}.
\end{equation}
Indeed, for the term $\frac{Y_{11}(T)}{T}$, we have $X_1(t) = \int_{0}^{t}
e^{-\theta(t-u)} dW_1(u)$, which can also be written as $X_1(t) = Z_1(t) -
e^{-\theta t} Z_1(0)$, where $Z_1(t) = \int_{-\infty}^{t} e^{-\theta(t-u)}
dW_1(u)$, we recall that the process $Z_1$ is ergodic, stationary and
Gaussian, therefore by the Ergodic theorem, the following convergences hold
as $T \rightarrow + \infty$, 
\begin{equation}  \label{ergo}
\frac{1}{T} \int_{0}^{T} X_1(t) dt = \frac{1}{T} \int_{0}^{T}Z_1(t) dt +%
\frac{Z_1(0)}{T}(1-e^{-\theta T}) \overset{\mathrm{a.s.}}{ \longrightarrow }
E[Z_1(0)] = 0, \text{ \ } \frac{1}{T} \int_{0}^{T} X_1^2(t) dt \overset{%
\mathrm{a.s.}}{ \longrightarrow } E[Z_1^2(0)] = \frac{1}{2 \theta}.
\end{equation}
Hence, we get as $T \rightarrow +\infty$: 
\begin{equation}  \label{estim-theta}
\frac{Y_{11}(T)}{T} = \frac{1}{T} \int_{0}^{T} X^{2}_1(u) du - \left(\frac{1%
}{T} \int_{0}^{T} X_1(u) du\right)^2 \overset{a.s.}{\longrightarrow} \frac{1%
}{2 \theta}.
\end{equation}
It follows that as $T \rightarrow +\infty$ $\sqrt{\frac{Y_{11}(T)}{T}}%
\overset{\mathrm{a.s.}}{ \longrightarrow } \frac{1}{\sqrt{2 \theta}} $ and
similarly we have $\sqrt{\frac{Y_{22}(T)}{T}}\overset{\mathrm{a.s.}}{
\longrightarrow } \frac{1}{\sqrt{2 \theta}}.$
\end{proof}

\subsection{Law of large numbers of $\protect\rho(T)$ under $H_a$ :}

\begin{theorem}
\label{LLN-rho} Under $H_a$, Yule's nonsense correlation $\rho(T)$
satisifies the following law of large numbers : 
\begin{equation*}
\rho(T) \overset{a.s}{\longrightarrow} r, \ \text{as} \ T \rightarrow +\infty
\end{equation*}
\end{theorem}

\begin{proof}
According to Proposition \ref{norm-denom}, equation (\ref{a.sdenom}), we
have 
\begin{equation}  \label{a.sdenom}
\sqrt{\frac{Y_{11}(T)}{T} \times \frac{Y_{22}(T)}{T}} \overset{\mathrm{a.s.}}%
{ \longrightarrow } \frac{1}{2 \theta}.
\end{equation}
It remains to proof that the numerator term satisfies : 
\begin{equation}  \label{LLN-num}
\frac{Y_{12}(T)}{T} \overset{a.s}{\longrightarrow} \frac{r}{2 \theta}, \ 
\text{as} \ T \rightarrow +\infty.
\end{equation}
From (\ref{decomp-num}), we have 
\begin{equation}  \label{num-as}
\frac{Y_{12}(T)}{T} = \frac{A_{r}(T)}{\sqrt{T}} + \frac{r}{2 \theta} + O(%
\frac{1}{\sqrt{T}}) - \bar{X_{1}}(T) \bar{X_{2}}(T).
\end{equation}
We have 
\begin{equation*}
\frac{A_r(T)}{\sqrt{T}} = \frac{c_1(r)}{T} \int_{0}^{T} I^{\mathcal{U}%
_1}_2(f^{\otimes 2}_u) du + \frac{c_2(r)}{T} \int_{0}^{T} I^{\mathcal{U}%
_0}_2(f^{\otimes 2}_u) du,
\end{equation*}
with $c_1(r) = \frac{r\sqrt{2}+ \sqrt{1-r^2}}{2}$ and $c_2(r) = \frac{r 
\sqrt{2} - \sqrt{1-r^2}}{2}$. We will start by proving that : for $i=0,1$, 
\begin{equation*}
\frac{1}{n} \int_{0}^{n}I^{\mathcal{U}_i}_2(f^{\otimes 2}_u) du \overset{a.s.%
}{\longrightarrow} 0, \ \text{as} \ n \rightarrow +\infty.
\end{equation*}
We have $\mathcal{U}_1 \overset{law}{=} \mathcal{U}_0 $, both terms can be
treated similarly: 
\begin{align*}
E\left[\left( \frac{1}{n} \int_{0}^{n}I^{\mathcal{U}_1}_2(f^{\otimes 2}_u)
du\right)^2\right] & = \frac{1}{n^2} \int_{0}^{n} \int_{0}^{n} E\left[I^{%
\mathcal{U}_1}_2(f^{\otimes 2}_u) I^{\mathcal{U}_1}_2(f^{\otimes 2}_v)\right]
du dv \\
& = \frac{1}{n^{2}} \int_{0}^{n} \int_{0}^{n} \langle f^{\otimes 2}_u,
f^{\otimes 2}_v \rangle du dv \\
& = \frac{1}{n^{2}} \int_{0}^{n} \int_{0}^{n} \left(\langle f_u, f_v
\rangle\right)^2 du dv \\
& = \frac{1}{n^{2}} \frac{1}{4 \theta^2} \int_{0}^{n} \int_{0}^{n} e^{-2
\theta (u+v)} [e^{2 \theta u \wedge v} -1]^2 du dv \\
& \leqslant \frac{1}{n^{2}} \frac{1}{4 \theta^2} \int_{0}^{n} \int_{0}^{n}
e^{-2 \theta |u-v|} du dv \\
& = \frac{1}{n^2} \frac{1}{2 \theta^2} \int_{0}^{n} \int_{0}^{u} e^{-2
\theta t} dt du \\
& = \frac{1}{n} \frac{1}{4 \theta^3} + \frac{1}{n^2 8 \theta^4} (e^{-2
\theta n} -1 ) =O(n^{-1}).
\end{align*}
Let $\varepsilon> 0$, $p >2$, then we can write using the hypercontractivity
property of multiple Wiener integrals (\ref{hypercontractivity}) : 
\begin{align*}
\sum\limits_{n =1}^{+ \infty} P \left(\left|\frac{A_r(n)}{\sqrt{n}} \right|
> \varepsilon \right) & = \sum\limits_{n =1}^{+ \infty} P \left(\left|\frac{%
c_1(r)}{n} \int_{0}^{n} I^{\mathcal{U}_1}_2(f^{\otimes 2}_u) du + \frac{%
c_2(r)}{n} \int_{0}^{n} I^{\mathcal{U}_0}_2(f^{\otimes 2}_u) du \right| >
\varepsilon \right) \\
& \leqslant \sum\limits_{n =1}^{+ \infty} P \left(\left| \frac{c_1(r)}{n}
\int_{0}^{n} I^{\mathcal{U}_1}_2(f^{\otimes 2}_u) du\right| > \frac{%
\varepsilon}{2}\right) + \sum\limits_{n =1}^{+ \infty} P \left(\left| \frac{%
c_2(r)}{n} \int_{0}^{n} I^{\mathcal{U}_0}_2(f^{\otimes 2}_u) du\right| > 
\frac{\varepsilon}{2}\right) \\
& \leqslant \frac{2^{p+1}(c^p_1(r) + c^p_2(r))}{\varepsilon^p} \times \sum
\limits_{n =1}^{+\infty} E\left[\left| \frac{1}{n} \int_{0}^{n} I^{\mathcal{U%
}_1}_2(f^{\otimes 2}_u) du\right|^2\right]^{p/2} \\
& \leqslant \frac{C(r,p)}{\varepsilon^p} \sum\limits_{n=1}^{+\infty} \frac{1%
}{n^{p/2}} < +\infty.
\end{align*}
It follows by Borel-Cantelli's Lemma that $\frac{A_r(n)}{\sqrt{n}} \overset{%
a.s.}{\longrightarrow} 0$, when $n \rightarrow +\infty$. By Lemma 3.3. of 
\cite{DEN}, we can conclude that we also have as $T \rightarrow +\infty$, 
\begin{equation*}
\frac{A_r(T)}{\sqrt{T}} \overset{a.s.}{\longrightarrow} 0.
\end{equation*}
Hence, using (\ref{ergo}) we have $\bar{X_1}(T) \bar{X_2}(T)
\longleftrightarrow 0$ as $T \rightarrow +\infty$. It follows therefore by (%
\ref{num-as}) that : 
\begin{equation*}
\frac{Y_{12}(T)}{T} \overset{a.s.}{\longrightarrow} \frac{r}{2 \theta}.
\end{equation*}
The desired result is obtained.
\end{proof}

\subsection{Gaussian fluctuations of $\protect\rho(T)$ under $H_a$:}

From (\ref{a.sdenom}), we will use of the following approximation for $\sqrt{%
T} \left(\rho(T)-r\right)$ for $T$ large : 
\begin{align*}
\sqrt{T} \left(\rho(T)-r\right) & \simeq \frac{\left(\frac{Y_{12}(T)}{\sqrt{T%
}} - \frac{r \sqrt{T}}{2 \theta}\right)}{\sqrt{\frac{Y_{11}(T)}{T} \times 
\frac{Y_{22}(T)}{T}}}
\end{align*}
It follows that, we can write for $T$ large, 
\begin{align*}
\sqrt{T}\frac{\sqrt{\theta}}{\sqrt{1+r^2}} \left( \rho(T)-r \right) & \simeq 
\frac{2 \theta^{3/2}}{\sqrt{1+r^2}}{\left(\frac{Y_{12}(T)}{\sqrt{T}} - \frac{%
r \sqrt{T}}{2 \theta}\right)} \slash 2\theta \sqrt{\frac{Y_{11}(T)}{T}
\times \frac{Y_{22}(T)}{T}}.
\end{align*}
Using the triangular property of $d_W$, Cauchy-Schwarz and Holder's
inequalities, we get 
\begin{align*}
&d_W\left(\frac{\sqrt{T}\sqrt{\theta}}{\sqrt{1+r^2}} \left( \rho(T)-r
\right), N\right) \\
& \leqslant d_W\left( \frac{2 \theta^{3/2}}{\sqrt{1+r^2}}{\left(\frac{%
Y_{12}(T)}{\sqrt{T}} - \frac{r \sqrt{T}}{2 \theta}\right)}, {N}\right) \\
& \text{ \ } \text{ \ } \text{ \ } + d_W\left( \frac{2 \theta^{3/2}}{\sqrt{%
1+r^2}}{\left(\frac{Y_{12}(T)}{\sqrt{T}} - \frac{r \sqrt{T}}{2 \theta}\right)%
} \slash 2\theta \sqrt{\frac{Y_{11}(T)}{T} \times \frac{Y_{22}(T)}{T}}, 
\frac{2 \theta^{3/2}}{\sqrt{1+r^2}}{\left(\frac{Y_{12}(T)}{\sqrt{T}} - \frac{%
r \sqrt{T}}{2 \theta}\right)} \right) \\
& \leqslant d_W\left( \frac{2 \theta^{3/2}}{\sqrt{1+r^2}}{\left(\frac{%
Y_{12}(T)}{\sqrt{T}} - \frac{r \sqrt{T}}{2 \theta}\right)}, {N}\right) \\
& \text{ \ } \text{ \ } \text{ \ }+ E\left[\frac{2 \theta^{3/2}}{\sqrt{1+r^2}%
}{\left(\frac{Y_{12}(T)}{\sqrt{T}} - \frac{r \sqrt{T}}{2 \theta}\right)} %
\slash 2\theta \sqrt{\frac{Y_{11}(T)}{T} \times \frac{Y_{22}(T)}{T}}
\left(1- 2 \theta \sqrt{\frac{Y_{11}(T)}{T} \times \frac{Y_{22}(T)}{T}}%
\right)\right] \\
& \leqslant d_W\left( \frac{2 \theta^{3/2}}{\sqrt{1+r^2}}{\left(\frac{%
Y_{12}(T)}{\sqrt{T}} - \frac{r \sqrt{T}}{2 \theta}\right)}, {N}\right) \\
&\text{ \ } \text{ \ } \text{ \ } + E\left[ \left(\frac{2 \theta^{3/2}}{%
\sqrt{1+r^2}}{\left(\frac{Y_{12}(T)}{\sqrt{T}} - \frac{r \sqrt{T}}{2 \theta}%
\right)} \slash 2\theta \sqrt{\frac{Y_{11}(T)}{T}\times \frac{Y_{22}(T)}{T}}
\right)^2\right]^{1/2} E\left[ \left( 1- 2 \theta \sqrt{\frac{Y_{11}(T)}{T} 
\frac{Y_{22}(T)}{T}}\right)^2\right]^{1/2} \\
& \leqslant d_W\left( \frac{2 \theta^{3/2}}{\sqrt{1+r^2}}{\left(\frac{%
Y_{12}(T)}{\sqrt{T}} - \frac{r \sqrt{T}}{2 \theta}\right)}, {N}\right) \\
&\text{ \ } \text{ \ } \text{ \ } + \| \frac{2 \theta^{3/2}}{\sqrt{1+r^2}}{%
\left(\frac{Y_{12}(T)}{\sqrt{T}} - \frac{r \sqrt{T}}{2 \theta}\right)}
\|_{L^4} \| 1 \slash 2 \theta \sqrt{\frac{Y_{11}(T)}{T} \frac{Y_{22}(T)}{T}}
\|_{L^4} \|1- 2 \theta \sqrt{\frac{Y_{11}(T)}{T} \frac{Y_{22}(T)}{T}}
\|_{L^2}.
\end{align*}
According to Theorem \ref{CLT-num}, there exists a constant $c(\theta, r)$
such that $d_W\left( \frac{2 \theta^{3/2}}{\sqrt{1+r^2}}{\left(\frac{%
Y_{12}(T)}{\sqrt{T}} - \frac{r \sqrt{T}}{2 \theta}\right)}, {N}%
\right)\leqslant \frac{c(\theta,r)}{\sqrt{T}}$, on the other hand, thanks to
Proposition \ref{norm-denom}, the term $\|1- 2 \theta \sqrt{\frac{Y_{11}(T)}{%
T} \frac{Y_{22}(T)}{T}} \|_{L^2} \leqslant \frac{c(\theta)}{\sqrt{T}}$.%
\newline
It remains to prove that both terms $\| 1 \slash  2 \theta \sqrt{\frac{%
Y_{11}(T)}{T} \frac{Y_{22}(T)}{T}} \|_{L^4} $ and $\| \frac{2 \theta^{3/2}}{%
\sqrt{1+r^2}}{\left(\frac{Y_{12}(T)}{\sqrt{T}} - \frac{r \sqrt{T}}{2 \theta}%
\right)} \|_{L^4}$ are finite. \newline
Using the decomposition (\ref{decomp-num}) and Minskowski's inequality and
the hypercontractivity property, we get for all $T > 0$, 
\begin{align*}
\|{\frac{Y_{12}(T)}{\sqrt{T}} - \frac{r \sqrt{T}}{2 \theta}} \|_{L^4} &
\leqslant \|A_r(T) \|_{L^4} + O(\frac{1}{\sqrt{T}}) + \sqrt{T} \| \bar{X}%
_{1}(T) \bar{X}_2(T)\|_{L^4} \\
& \leqslant 3 \|A_r(T) \|_{L^2} + O(\frac{1}{\sqrt{T}}) + r \sqrt{T} \| \bar{%
X}^2_{1}(T)\|_{L^4} + \sqrt{1-r^2} \sqrt{T} \| \bar{X}_{1}(T) \bar{X}_0(T)
\|_{L^4} \\
& \leqslant 3 \|A_r(T) \|_{L^2} + O(\frac{1}{\sqrt{T}}) + 7r\sqrt{T} E[\bar{X%
}^2_{1}(T)] + 3\sqrt{T} \sqrt{1-r^2} \| \bar{X}_{1}(T)\|_{L^2} \| \bar{X}%
_{0}(T)\|_{L^2} \\
& \leqslant 3 \|A_r(T) \|_{L^2} + O(\frac{1}{\sqrt{T}}) \leqslant
C(\theta,r) .
\end{align*}
where we used inequality (\ref{normbarX}) and Proposition \ref{estim-var}.
It follows that $\sup\limits_{T >0} \| \frac{2 \theta^{3/2}}{\sqrt{1+r^2}}{%
\left(\frac{Y_{12}(T)}{\sqrt{T}} - \frac{r \sqrt{T}}{2 \theta}\right)}
\|_{L^4} < \infty.$ For the term $\| 1 \slash  2 \theta \sqrt{\frac{Y_{11}(T)%
}{T} \frac{Y_{22}(T)}{T}} \|_{L^4} $, making use of the notation $\bar{Y}%
_{ii}(T):= 2\theta \frac{Y_{ii}(T)}{T}$, $i=1,2$, it is sufficient to show
that there exists $T_0>0$ such that $\sup \limits_{T \geq T_0 }E[\bar{Y}%
_{ii}(T)^{-4}] < +\infty$, for $i=1,2.$ \newline

From equation (\ref{estim-theta}), it's easy to derive an estimator of the
parameter $\theta$ if the latter is unknown, in fact we showed that for $%
i=1,2$ :

\begin{equation}  \label{estimateur-theta}
\tilde{\theta}_{T} := \frac{1}{2}\left(\frac{Y_{ii}(T)}{T} \right)^{-1}:=%
\frac{1}{2} \left(\frac{1}{T}\int_{0}^{T} X_i(t)^2 dt - \left(\frac{1}{T}
\int_{0}^{T} X_i(t) dt\right)^2 \right)^{-1},
\end{equation}
is strongly consistent. Moreover, in the reference \cite{XY} it is proved
that $\tilde{\theta}_{T}$ is Gaussian, more precisely, we have 
\begin{equation*}
\sqrt{T} \left( \frac{1}{2}\left(\frac{Y_{ii}(T)}{T} \right)^{-1} - \theta
\right) \xrightarrow[T \rightarrow +\infty]{\mathcal{L}} \mathcal{N}(0,2
\theta).
\end{equation*}
Therefore, using the Delta method, we can conclude that for $i=1,2$ 
\begin{equation}  \label{normal-barY}
\sqrt{T} \left( \bar{Y}_{ii}(T) - 1 \right) \xrightarrow[T \rightarrow
+\infty]{\mathcal{L}} \mathcal{N}(0, \frac{2}{\theta}).
\end{equation}
We can write 
\begin{equation*}
E[\bar{Y}^{-4}_{11}(T)] = E[\bar{Y}^{-4}_{11}(T) \mathbf{1}_{\left\{|\bar{Y}%
_{11}(T) -1| \geq \frac{1}{\sqrt{T}} \right\}}] + E[\bar{Y}^{-4}_{11}(T) 
\mathbf{1}_{\left\{|\bar{Y}_{11}(T) -1| < \frac{1}{\sqrt{T}}\right\}}]
\end{equation*}
For the second expectation $E[\bar{Y}^{-4}_{11}(T) \mathbf{1}_{\left\{|\bar{Y%
}_{11}(T) -1| < \frac{1}{\sqrt{T}}\right\}}]$,we have $1-\frac{1}{ \sqrt{T}}
< \bar{Y}_{11}(T)< 1+ \frac{1}{\sqrt{T}}$ a.s. we can say that for all $%
T>T_0 := 1$, $\bar{Y}_{11}(T) >0$ a.s. thus it's bounded away from 0. Hence $%
\sup \limits_{T \geq T_0}E[\bar{Y}^{-4}_{11}(T) \mathbf{1}_{\left\{|\bar{Y}%
_{11}(T) -1| < \frac{1}{\sqrt{T}}\right\}}] < +\infty.$ For the first
expectation $E[\bar{Y}^{-4}_{11}(T) \mathbf{1}_{\left\{|\bar{Y}_{11}(T) -1|
\geq \frac{1}{\sqrt{T}} \right\}}] $, using the (\ref{normal-barY}), we can
say that for $T > \frac{2 \pi}{\theta}$, we have 
\begin{align*}
E[\bar{Y}^{-4}_{11}(T) \mathbf{1}_{\left\{|\bar{Y}_{11}(T) -1| \geq \frac{1}{%
\sqrt{T}} \right\}}] & \sim \int_{\frac{ \sqrt{\theta}}{\sqrt{2}}}^{+\infty}
\left( \frac{\sqrt{2}}{\sqrt{\theta}} \frac{1}{\sqrt{T}} z+1\right)^{-4}
e^{- \frac{z^2}{2}}dz + \int_{\frac{ \sqrt{\theta}}{\sqrt{2}}}^{+\infty} |1-%
\frac{\sqrt{2}}{\sqrt{\theta}} \frac{1}{\sqrt{T}} z |^{-4} e^{-\frac{z^2}{2}%
} dz < +\infty.
\end{align*}
Therefore, for $T >(T_0 \vee \frac{2 \pi}{\theta})$, $E[\bar{Y}%
_{11}(T)^{-4}] < +\infty$, therefore following theorem follows.

\begin{theorem}
\label{CLT-rho} There exists a constant $C(\theta,r)$ depending on $\theta$
and $r$ such that 
\begin{equation*}
d_{W}\left( \sqrt{T}\frac{\sqrt{\theta}}{\sqrt{1+r^2}}\left(\rho(T) -
r\right), \mathcal{N}\left(0, 1\right)\right) \leqslant \frac{C(\theta,r)}{%
\sqrt{T}}
\end{equation*}
In particular, 
\begin{equation*}
\sqrt{T} \left(\rho(T) - r \right) \overset{\mathrm{\mathcal{L}}}{
\longrightarrow } \mathcal{N}\left(0,\frac{1+r^2}{\theta}\right) \text{ \ } 
\text{as} \text{ \ } T \rightarrow +\infty.
\end{equation*}
\end{theorem}

\begin{remark}
Notice, in the particular case when $X_1$ and $X_2$ are independent that is $%
r =0$, we find the CLT that we proved in Theorem 3.8 of \cite{DEV}, that is
when $T \rightarrow +\infty$, then 
\begin{equation*}
\sqrt{\theta} \sqrt{T} \rho(T) \overset{\mathcal{L}}{\longrightarrow } 
\mathcal{N}(0,1).
\end{equation*}
Moreover, the rate of convergence in $d_W$ metric is even better that the
bound we found under $H_0$ is \cite{DEV}, Theorem 3.8 which was of the order
of $\frac{\ln(T)}{\sqrt{T}}$, while according to Theorem \ref{CLT-rho}, we
get when $r=0$ : 
\begin{equation*}
d_{W}\left( \sqrt{T}\sqrt{\theta}\rho(T), \mathcal{N}\left(0,
1\right)\right) \leqslant \frac{C(\theta)}{\sqrt{T}}.
\end{equation*}
\end{remark}

\section{Some statistical applications of the Gaussian asymptotics of $%
\protect\rho (T)$ under $H_{a}$ \label{TESTING}}

One possible scenario, that may happen in practice, is when the paths $X_1$
and $X_2$ are correlated but the value of the correlation $r$ is \textbf{%
unknown}. In this case, since Yule's nonsense correlation $\rho(T)$
satisfies a LLN, see Theorem \ref{LLN-rho} under $H_a$, $\rho(T)$ approaches
the value of the true correlation $r$ when the horizon $T$ is large. We can
therefore, consider $\rho(T)$ as a strongly consistent estimator of the
parameter $r$. \newline
Moreover, using the Gaussian fluctuations of $\rho(T)$ under $H_a$, Theorem %
\ref{CLT-rho} in addition to Slutsky's Lemma, we can write

\begin{equation*}
\frac{\sqrt{T} \sqrt{\theta}}{\sqrt{1+ \rho^2(T)}} \left(\rho(T) - r \right) 
\overset{\mathrm{\mathcal{L}}}{ \longrightarrow } \mathcal{N}%
\left(0,1\right) \text{ \ } \text{as} \text{ \ } T \rightarrow +\infty.
\end{equation*}
We can derive from this an asymptotic confidence interval level $(1-\alpha)$
of the parameter $r$ which is the following : 
\begin{equation*}
I_{\theta}(T) := \left[ \rho(T) - \frac{\sqrt{1+ \rho^2(T)}}{\sqrt{\theta} 
\sqrt{T}} q_{\alpha/2}, \rho(T) + \frac{\sqrt{1+ \rho^2(T)}}{\sqrt{\theta} 
\sqrt{T}} q_{\alpha /2}\right],
\end{equation*}
where $q_{\alpha /2}$ is the upper quantile of order of the standard
Gaussian law $\mathcal{N}(0,1)$.\newline
It may happen in practice, the drift parameter $\theta$ which is common for $%
X_1$ and $X_2$ is also \textbf{unknown} as well, in this case, Theorem \ref%
{CLT-rho} in addition to Slutsky's Lemma, we can write

\begin{equation*}
\frac{\sqrt{T} \sqrt{\tilde{\theta}_T}}{\sqrt{1+ \rho^2(T)}} \left(\rho(T) -
r \right) \overset{\mathrm{\mathcal{L}}}{ \longrightarrow } \mathcal{N}%
\left(0,1\right) \text{ \ } \text{as} \text{ \ } T \rightarrow +\infty.
\end{equation*}
where $\tilde{\theta}_T$ is the estimator of theta defined in (\ref%
{estim-theta}). In this case, an asymptotic confidence interval level $%
(1-\alpha)$ of the parameter $r$ is given by 
\begin{equation*}
I(T) := \left[ \rho(T) - \frac{\sqrt{1+ \rho^2(T)}}{\sqrt{\tilde{\theta_T}} 
\sqrt{T}} q_{\alpha /2}, \rho(T) + \frac{\sqrt{1+ \rho^2(T)}}{\sqrt{\tilde{%
\theta_T} \sqrt{T}}} q_{\alpha /2}\right],
\end{equation*}

\subsection{Testing independence of $X_{1}$ and $X_{2}$: Rejection region
and power of the test\label{Testwithrho}}

The aim of this section is to build a statistical test of independence (or
dependence) of the pair of processes $(X_1,X_2)$. That is we propose a test
for the following hypothesis

\begin{center}
$H_0 :$ $(X_1)$ and $(X_2)$ are independent.

Versus

$H_a:$ $(X_1)$ and $(X_2)$ are correlated for some $r=cor(W_1,W_2) \in
[-1,1] \backslash \{0\}.$
\end{center}

based on the statistic $\rho(T)$ observed on the time interval $[0,T]$.
Using the results that we found in the previous section, we will define the
rejection regions and study the power of the test. \newline
Let us fix a \textbf{significance level} $\alpha \in (0,1)$. We proved in 
\cite{DVE}, that under $H_0$, the Yule statistic $\rho(T)$ satisfies the
following CLT : 
\begin{equation}  \label{TCL0}
\sqrt{T} \rho(T) \overset{\mathrm{\mathcal{L}}}{ \longrightarrow } \mathcal{N%
}\left(0, \frac{1}{\theta}\right) \text{ \ } \text{as} \text{ \ } T
\rightarrow +\infty
\end{equation}
Therefore, a natural test of independence of $(X_1)$ and $(X_2)$ is to
reject independence if $\left\{\sqrt{T} |\rho(T)| > c_{\alpha} \right\}$, $%
c_{\alpha}$ a threshold depending on $\alpha$, that we will determine. 
\newline
On the other hand, by definition of type I error and \ref{TCL0}, we can
write that for $T$ large, we have 
\begin{align*}
\alpha & \approx \mathbb{P}_{H_0}\left(\sqrt{T} |\rho(T)| > c_{\alpha}\right)
\\
& = \mathbb{P}_{H_0}\left(\sqrt{T} \sqrt{\theta}|\rho(T)| > \sqrt{\theta}
c_{\alpha}\right)
\end{align*}
We infer that a natural rejection regions $\mathcal{R}_{\alpha}$ are of the
form : 
\begin{equation*}
\mathcal{R}_{\alpha} : = \left\{\sqrt{T} |\rho(T)| > \frac{q_{\alpha/2}}{%
\sqrt{\theta}} \right\}.
\end{equation*}
where $q_{\alpha/2}$ is the upper quantile of standard normal distribution.
The following proposition gives an estimate of type II error based on
Theorem \ref{CLT-rho}.

\begin{proposition}
\label{estim-t2} Fix $\alpha \in (0,1)$, then for $T$ large and $r \in
[-1,1] \backslash \{0\}$. Then, there exists a constant $C(\theta,r,\alpha)$
depending on $\theta$, $\alpha$ and $r$ such that we have : 
\begin{equation*}
\beta =\mathbb{P}_{H_a}\left[ \sqrt{T} |\rho(T)| \leqslant \frac{q_{\alpha/2}%
}{\sqrt{\theta}} \right] \leqslant \frac{C(\theta,r,\alpha)}{T^{1/4}}.
\end{equation*}
\end{proposition}

\begin{proof}
Another, estimate for the type II error which is also a direct consequence
of the rate of convergence that we found in Theorem \ref{CLT-rho} is the
following : \newline
Denote $Z_T : = \frac{\sqrt{T}}{\sigma_{r,\theta}} \left(\rho(T) - r\right) $
and $F_{Z_T}(.)$ its cumulative distribution function and $c_{\alpha} = 
\frac{q_{\alpha/2}}{\sqrt{\theta}}$ Then, under $H_a$, we have 
\begin{align*}
\beta & = \mathbb{P}_{H_a}\left[ \sqrt{T} |\rho(T)| \leqslant c_{\alpha}%
\right] \\
& = \mathbb{P}_{H_a}\left[ |\sqrt{T}(\rho(T)-r) + r\sqrt{T}| \leqslant
c_{\alpha}\right] \\
& = \mathbb{P}_{H_a}\left[ \left|Z_T + \frac{r\sqrt{T}}{\sigma_{r,\theta}}
\right| \leqslant \frac{c_{\alpha}}{\sigma_{r,\theta}}\right] = \mathbb{P}%
_{H_a}\left[ \frac{-c_{\alpha} -r\sqrt{T} }{\sigma_{r,\theta}} \leqslant Z_T
\leqslant \frac{c_{\alpha} -r\sqrt{T} }{\sigma_{r,\theta}}\right] \\
& = F_{Z_T}\left(\frac{c_{\alpha} -r\sqrt{T} }{\sigma_{r,\theta}}\right) -
F_{Z_T}\left(\frac{-c_{\alpha} -r\sqrt{T} }{\sigma_{r,\theta}}\right).
\end{align*}
Thus the following upper bound holds : 
\begin{align}  \label{estim-beta}
\mathbb{P}_{H_a}\left[ \sqrt{T} |\rho(T)| \leqslant c_{\alpha}\right] &
\leqslant \left|F_{Z_T}\left(\frac{c_{\alpha} -r\sqrt{T} }{\sigma_{r,\theta}}%
\right) - \mathbf{\phi}\left(\frac{c_{\alpha} -r\sqrt{T} }{\sigma_{r,\theta}}%
\right) \right| + \left|\mathbf{\phi}\left(\frac{c_{\alpha} -r\sqrt{T} }{%
\sigma_{r,\theta}}\right) - \mathbf{\phi}\left(\frac{- c_{\alpha} -r\sqrt{T} 
}{\sigma_{r,\theta}}\right)\right|  \notag \\
& \text{ \ } \text{ \ } + \left|F_{Z_T}\left(\frac{-c_{\alpha} -r\sqrt{T} }{%
\sigma_{r,\theta}}\right) - \mathbf{\phi}\left(\frac{-c_{\alpha} -r\sqrt{T} 
}{\sigma_{r,\theta}}\right) \right|
\end{align}
Moreover, we have the following estimates for $T$ large enough: 
\begin{eqnarray}
\left|\mathbf{\phi}\left(\frac{c_{\alpha} -r\sqrt{T} }{\sigma_{r,\theta}}%
\right) - \mathbf{\phi}\left(\frac{- c_{\alpha} -r\sqrt{T} }{%
\sigma_{r,\theta}}\right)\right| \leqslant &&\left\{ 
\begin{array}{ll}
\frac{2 c_\alpha}{\sigma_{r,\theta} \sqrt{2 \pi}} e^{-\left(\frac{c_\alpha- r%
\sqrt{T}}{\sigma_{r,\theta}}\right)^2} & \mbox{ if } r \in ]0,1], \\ 
~~ &  \\ 
\frac{2 c_\alpha}{\sigma_{r,\theta} \sqrt{2 \pi}} e^{-\left(\frac{-c_\alpha-
r\sqrt{T}}{\sigma_{r,\theta}}\right)^2} & \mbox{ if } r \in [-1,0[.%
\end{array}
\right.
\end{eqnarray}

Moreover, we have : 
\begin{align*}
& \left|F_{Z_T}\left(\frac{c_{\alpha} -r\sqrt{T} }{\sigma_{r,\theta}}\right)
- \mathbf{\phi}\left(\frac{c_{\alpha} -r\sqrt{T} }{\sigma_{r,\theta}}\right)
\right| + \left|F_{Z_T}\left(\frac{-c_{\alpha} -r\sqrt{T} }{\sigma_{r,\theta}%
}\right) - \mathbf{\phi}\left(\frac{-c_{\alpha} -r\sqrt{T} }{%
\sigma_{r,\theta}}\right) \right| \\
& \leqslant 2 d_{Kol} \left(Z_T, N(0,1)\right) \leqslant 4 \sqrt{ d_{W}
\left(Z_T, N(0,1)\right)} \leqslant \frac{C(\theta,r)}{T^{1/4}}.
\end{align*}
That is, type II error for this test has the following estimate for any $T$
large : 
\begin{equation*}
\beta = \mathbb{P}_{H_a}\left[ \sqrt{T} |\rho(T)| \leqslant c_{\alpha}\right]
\leqslant \frac{C(\theta, r , \alpha)}{T^{1/4}}.
\end{equation*}
where $C(\theta, r , \alpha) = C(\theta, r) + \frac{2 c_\alpha}{%
\sigma_{r,\theta} \sqrt{2\pi}}.$
\end{proof}

A direct consequence of the previous proposition, is that the above test of
hypothesis is asymptotically powerful.

\begin{corollary}
\label{power} For $\alpha \in (0,1)$ fixed, then for $T$ large and $r \in
[-1,1] \backslash \{0\}$, we have under $H_a$ and for $T$ large, we have 
\begin{equation*}
\mathbb{P}_{H_a}\left[ |\sqrt{T} \rho(T)| > \frac{q_{\alpha/2}}{\sqrt{\theta}%
}\right] \geq 1 - \frac{C(\theta,r,\alpha)}{T^{1/4}}.
\end{equation*}
In particular, this test of hypothesis is asymptotically powerful : 
\begin{equation*}
\mathbb{P}_{H_a}\left[\sqrt{\theta}\sqrt{T} |\rho(T)| > {q_{\alpha/2}}\right]
\underset{T \rightarrow +\infty}{ \longrightarrow } 1.
\end{equation*}
where $q_{\alpha/2}$ is the upper $\alpha/2$ quantile of standard normal
distribution.
\end{corollary}

\begin{remark}
Another scenario, one could consider in this framework is what the rejection
regions will be if the drift parameter $\theta$ is unknown? \newline
In this case,we will make use of the estimator (\ref{estimateur-theta}) of
the parameter $\theta$: 
\begin{equation*}
\tilde{\theta}_T := \frac{1}{2}\left(\frac{Y_{ii}(T)}{T} \right)^{-1}:=\frac{%
1}{2} \left(\frac{1}{T}\int_{0}^{T} X_i(t)^2 dt - \left(\frac{1}{T}
\int_{0}^{T} X_i(t) dt\right)^2 \right)^{-1} \ i=1,2,
\end{equation*}%
. We infer using Slutsky's lemma that in this case, we can consider for a a
significance level $\alpha \in (0,1)$, the following rejection regions : 
\begin{equation*}
\tilde{\mathcal{R}}_{\alpha} : = \left\{\sqrt{T \tilde{\theta}_T} |\rho(T)|
> {q_{\alpha/2}} \right\}.
\end{equation*}
Proposition \ref{estim-t2} and Corollary \ref{power} can be extended easily
for this alternative test of hypothesis, thus it's also asymptotically
powerful.
\end{remark}

\subsection{A test of independence using the numerator :\label{Testwithnum}}

We showed that under $H_a$, we have the existence of a constant $C(\theta,r)
> 0$, such that : 
\begin{equation*}
d_W\left(\frac{1}{\sigma_{r,\theta}} \left(\frac{Y_{12}(T)}{\sqrt{T}} - 
\frac{r \sqrt{T}}{2 \theta}\right),\mathcal{N}(0,1)\right) \leqslant \frac{%
C(\theta,r)}{\sqrt{T}}.
\end{equation*}
where $\sigma_{r,\theta} := \left(\frac{1}{2 \theta^3} \left(\frac{1}{2}+%
\frac{r^2}{2}\right)\right)^{1/2}.$ In particular, under $H_0$, we have as $%
T \rightarrow +\infty$, 
\begin{equation*}
\frac{Y_{12}(T)}{\sqrt{T}} \overset{\mathcal{L}}{\longrightarrow} \mathcal{N}%
\left(0,\frac{1}{4 \theta^{3}}\right).
\end{equation*}
We can therefore, consider the numerator itself as a statistic of our test
and therefore to reject independence if $\left\{\left| \frac{Y_{12}(T)}{%
\sqrt{T}} \right| > c_{\alpha}\right\},$ where $\alpha \approx \mathbb{P}%
_{H_0}\left( \left| \frac{Y_{12}(T)}{\sqrt{T}} \right| > c_{\alpha}\right) = 
\mathbb{P}_{H_0}\left( \left| \frac{Y_{12}(T)}{\sqrt{T}} \right| > \frac{%
q_{\alpha/2}}{2 \theta^{3/2}}\right).$

We infer that another natural rejection regions $\mathcal{R}_{\alpha}$ are
of the form : 
\begin{equation*}
\mathcal{R}_{\alpha} : = \left\{ \left| \frac{Y_{12}(T)}{\sqrt{T}} \right|> 
\frac{q_{\alpha/2}}{2 \theta^{3/2}} \right\}.
\end{equation*}
We have the following estimates for type II error for this test.

\begin{proposition}
\label{estim-num} Fix $\alpha \in (0,1)$ and $r \in [-1,1] \backslash \{0\}$%
. Then, there exists a constant $C(\theta,\alpha,r)$ depending on $\theta$, $%
r$ and $\alpha$ such that for all $T$ large, we have : 
\begin{equation}  \label{first-estimate}
\beta =\mathbb{P}_{H_a}\left[ \left|\frac{Y_{12}(T)}{\sqrt{T}} \right|
\leqslant \frac{q_{\alpha/2}}{2 \theta^{3/2}} \right] \leqslant C(\theta,
\alpha, r) \times \frac{\ln(T)}{\sqrt{T}}.
\end{equation}
\end{proposition}

A direct consequence of the previous proposition, is that the above test of
hypothesis is asymptotically powerful.

\begin{corollary}
For $\alpha \in (0,1)$ fixed, then for $T$ large and $r \in [-1,1]
\backslash \{0\}$, we have under $H_a$ and for $T$ large, we have 
\begin{equation*}
\mathbb{P}_{H_a}\left[ \left|\frac{Y_{12}(T)}{\sqrt{T}} \right| > \frac{%
q_{\alpha/2}}{2 \theta^{3/2}}\right] \geq 1 - C(\theta, \alpha,r) \times 
\frac{\ln(T)}{\sqrt{T}}.
\end{equation*}
In particular, this test of hypothesis is asymptotically powerful : 
\begin{equation*}
\mathbb{P}_{H_a}\left[ 2 \theta^{3/2} \left|\frac{Y_{12}(T)}{\sqrt{T}}
\right| > {q_{\alpha/2}}\right] \underset{T \rightarrow +\infty}{
\longrightarrow } 1.
\end{equation*}
where $q_{\alpha/2}$ is the upper $\alpha/2$ quantile of standard normal
distribution.
\end{corollary}

\begin{proof}
(of Proposition \ref{estim-num}) Denote $N_T : = \frac{1}{\sigma_{r,\theta}}
\left( \frac{Y_{12}(T)}{\sqrt{T}} - \frac{r \sqrt{T}}{2 \theta}\right) $ and 
$F_{N_T}(.)$ its cumulative distribution function and $c_{\alpha} = \frac{%
q_{\alpha/2}}{2 \theta^{3/2}}$ Then, under $H_a$, we have 
\begin{align*}
\beta & = \mathbb{P}_{H_a}\left[ \left| \frac{Y_{12}(T)}{\sqrt{T}} \right|
\leqslant c_{\alpha}\right] \\
& = \mathbb{P}_{H_a}\left[ \frac{1}{\sigma_{r,\theta}} \left| \left(\frac{%
Y_{12}(T)}{\sqrt{T}} -\frac{r \sqrt{T}}{2\theta} \right) + \frac{r \sqrt{T}}{%
2 \theta}\right|\leqslant \frac{c_{\alpha}}{\sigma_{r,\theta}}\right] \\
& = \mathbb{P}_{H_a}\left[ \left|N_T + \frac{r\sqrt{T}}{\sigma_{r,\theta}}
\right| \leqslant \frac{c_{\alpha}}{\sigma_{r,\theta}}\right] = \mathbb{P}%
_{H_a}\left[ \frac{-c_{\alpha}}{\sigma_{r,\theta}} - \frac{r \sqrt{T}}{2
\theta \sigma_{r,\theta}} \leqslant N_T \leqslant \frac{c_{\alpha}}{%
\sigma_{r,\theta}} - \frac{r \sqrt{T}}{2 \theta \sigma_{r,\theta}} \right] \\
& = F_{N_T}\left(\frac{c_{\alpha}}{\sigma_{r,\theta}} - \frac{r \sqrt{T}}{2
\theta \sigma_{r,\theta}} \right) - F_{N_T}\left(\frac{-c_{\alpha}}{%
\sigma_{r,\theta}} - \frac{r \sqrt{T}}{2 \theta \sigma_{r,\theta}}\right).
\end{align*}
Thus, the following upper bound holds : 
\begin{align}  \label{estimate2-error}
\mathbb{P}_{H_a}\left[ \left| \frac{Y_{12}(T)}{\sqrt{T}} \right| \leqslant
c_{\alpha}\right] & \leqslant \left|F_{N_T}\left(\frac{c_{\alpha}}{%
\sigma_{r,\theta}} - \frac{r \sqrt{T}}{2 \theta \sigma_{r,\theta}}\right) - 
\mathbf{\phi}\left(\frac{c_{\alpha}}{\sigma_{r,\theta}} - \frac{r \sqrt{T}}{%
2 \theta \sigma_{r,\theta}} \right) \right| + \left|\mathbf{\phi}\left(\frac{%
c_{\alpha}}{\sigma_{r,\theta}} - \frac{r \sqrt{T}}{2 \theta \sigma_{r,\theta}%
}\right) - \mathbf{\phi}\left(\frac{-c_{\alpha}}{\sigma_{r,\theta}} - \frac{%
r \sqrt{T}}{2 \theta \sigma_{r,\theta}}\right)\right|  \notag \\
& \text{ \ } \text{ \ } + \left|F_{N_T}\left(\frac{-c_{\alpha}}{%
\sigma_{r,\theta}} - \frac{r \sqrt{T}}{2 \theta \sigma_{r,\theta}}\right) - 
\mathbf{\phi}\left(\frac{-c_{\alpha}}{\sigma_{r,\theta}} - \frac{r \sqrt{T}}{%
2 \theta \sigma_{r,\theta}}\right) \right|
\end{align}
The following decomposition of the numerator follows from equality (\ref%
{decomp-num1}) 
\begin{align}  \label{decomp-num}
\frac{1}{\sigma_{r,\theta}} \left( \frac{Y_{12}(T)}{\sqrt{T}} - \frac{r 
\sqrt{T}}{2 \theta} \right) & = \frac{c_{1}(r)}{\sigma_{r,\theta}} I^{%
\mathcal{U}_1}_2\left(\frac{1}{\sqrt{T}} \int_{0}^{T} f^{\otimes 2}_u
du\right) + \frac{c_{2}(r)}{\sigma_{r,\theta}} I^{\mathcal{U}_0}_2\left(%
\frac{1}{\sqrt{T}} \int_{0}^{T} f^{\otimes 2}_u du\right)  \notag \\
& \ \ \ \ \ - \frac{r}{4 \theta^2}\frac{(1-e^{-2 \theta T})}{%
\sigma_{r,\theta}\sqrt{T}} - \frac{\sqrt{T}}{\sigma_{r,\theta}} \bar{X}_1(T) 
\bar{X}_2(T).
\end{align}

\begin{proposition}
\label{two-sided} Consider a two-sided Brownian motion $\{W(t), t \in 
\mathbb{R} \}$, constructed from $\mathcal{U}_1$ and $\mathcal{U}_0$ as
follows : 
\begin{equation*}
W(t) = \mathcal{U}_1(t) \mathbf{1}_{\{t \geq 0\}} + \mathcal{U}_0(-t) 
\mathbf{1}_{\{t < 0\}}
\end{equation*}
Then the following equalities hold a.s. 
\begin{align*}
& \frac{c_1(r)}{\sqrt{T}} \int_{0}^{T} I^{\mathcal{U}_1}_2(f^{\otimes 2}_u)
du + \frac{c_2(r)}{\sqrt{T}} \int_{0}^{T} I^{\mathcal{U}_0}_2(f^{\otimes
2}_u) du \\
& \overset{\mathrm{"a.s."}}{ = } I^{W}_2 \left( \frac{1}{\sqrt{T}}
\int_{0}^{T} \left(c_1(r) f^{\otimes 2}_u + c_2(r) \bar{\bar{f}}^{\otimes
2}_u \right)du \right) \\
& \overset{\mathrm{"a.s."}}{ = } I^{W}_2(h_T) + I^{W}_2(g_T),
\end{align*}
where 
\begin{equation*}
\bar{\bar{f}}_u(x) = -f_u(-x)= -e^{-\theta(u+x)} \mathbf{1}_{[-u,0]}(x).
\end{equation*}
\begin{eqnarray}  \label{kernels}
\left\{ 
\begin{array}{ll}
h_T(t,s) & = \frac{1}{2 \theta} \frac{1}{\sqrt{T}} \left[c_1(r) \mathbf{1}%
_{[0,T]^2}(t,s) + c_2(r) \mathbf{1}(t,s)_{[-T,0]^2}\right] e^{-\theta |t-s|}.
\\ 
&  \\ 
g_T(t,s) & = \frac{1}{2 \theta} \frac{1}{\sqrt{T}} \left[c_1(r) \mathbf{1}%
_{[0,T]^2}(t,s) + c_2(r) \mathbf{1}(t,s)_{[-T,0]^2}\right] e^{-2 \theta T}
e^{\theta |t-s|}.%
\end{array}
\right.
\end{eqnarray}
with : 
\begin{equation*}
c_1(r) = \frac{r\sqrt{2}}{2} + \frac{\sqrt{1-r^2}}{2}, \ \ \ c_2(r) = \frac{%
r \sqrt{2}}{2} - \frac{\sqrt{1-r^2}}{2}.
\end{equation*}
\end{proposition}

\begin{proof}
We have for $t,s \in [0,T]$, 
\begin{align*}
\frac{1}{\sqrt{T}}\int_{0}^{T} f^{\otimes 2}_u(t,s) du & = \frac{1}{\sqrt{T}}
\int_{0}^{T} e^{- \theta(u-t)} e^{-\theta(u-s)} \mathbf{1}_{[0,u]}(t) 
\mathbf{1}_{[0,u]}(s) du \\
& = \frac{1}{\sqrt{T}} e^{\theta(t+s)} \int_{t \vee s }^{T} e^{-2 \theta u}
du \mathbf{1}_{[0,T]}(t) \mathbf{1}_{[0,T]}(s) \\
& =\frac{1}{\sqrt{T}} e^{\theta(t+s)} \frac{1}{2 \theta} \left[ e^{-2 \theta
t\vee s} - e^{-2 \theta T}\right] \mathbf{1}_{[0,T]}(t) \mathbf{1}_{[0,T]}(s)
\end{align*}
It follows by linearity of multiple Wiener integrals, we have : 
\begin{align*}
I^{\mathcal{U}_1}_2\left(\frac{1}{\sqrt{T}} \int_{0}^{T} f^{\otimes 2}_u
du\right) & = \frac{1}{2 \theta} \frac{1}{\sqrt{T}} \int_{0}^{T}
\int_{0}^{T} e^{-2 \theta t\vee s} e^{\theta(t+s)} d\mathcal{U}_1(t) d%
\mathcal{U}_1(s) - \frac{1}{2 \theta} \frac{1}{\sqrt{T}} \int_{0}^{T}
\int_{0}^{T} e^{-2 \theta T}e^{\theta(t+s)} d\mathcal{U}_1(t) d\mathcal{U}%
_1(s), \\
& = \frac{1}{2 \theta} \frac{1}{\sqrt{T}} \int_{0}^{T} \int_{0}^{T} e^{-2
\theta t\vee s} e^{\theta(t+s)} dW(t) dW(s) - \frac{1}{2 \theta} \frac{1}{%
\sqrt{T}} \int_{0}^{T} \int_{0}^{T} e^{-2 \theta T}e^{\theta(t+s)} dW(t)
dW(s).
\end{align*}
On the other hand, using a change of variable $s^{\prime} = -s$, $t^{\prime}
= -t$, we get: 
\begin{align*}
& I^{\mathcal{U}_0}_2\left(\frac{1}{\sqrt{T}} \int_{0}^{T} f^{\otimes 2}_u
du\right) \\
& = \frac{1}{2 \theta} \frac{1}{\sqrt{T}} \int_{0}^{T} \int_{0}^{T} e^{-2
\theta t\vee s} e^{\theta(t+s)} d\mathcal{U}_1(t) d\mathcal{U}_1(s) - \frac{1%
}{2 \theta} \frac{1}{\sqrt{T}} \int_{0}^{T} \int_{0}^{T} e^{-2 \theta
T}e^{\theta(t+s)} d\mathcal{U}_0(t) d\mathcal{U}_0(s), \\
& = \frac{1}{2 \theta} \frac{1}{\sqrt{T}} \int_{-T}^{0} \int_{-T}^{0} e^{-
\theta t^{\prime}} e^{-\theta s^{\prime}} e^{2 \theta (t^{\prime} \wedge
s^{\prime})} d\mathcal{U}_0(-t^{\prime}) d\mathcal{U}_0(-s^{\prime})- \frac{1%
}{2 \theta} \frac{1}{\sqrt{T}} \int_{-T}^{0} \int_{-T}^{0} e^{-2 \theta
T}e^{-\theta(t^{\prime}+s^{\prime})} d\mathcal{U}_0(-t^{\prime}) d\mathcal{U}%
_0(-s^{\prime}) \\
& = \frac{1}{2 \theta} \frac{1}{\sqrt{T}} \int_{-T}^{0} \int_{-T}^{0}
e^{-\theta |t-s|} dW(t) dW(s) - \frac{1}{2 \theta} \frac{1}{\sqrt{T}}
\int_{-T}^{0} \int_{-T}^{0} e^{-2 \theta T} e^{ -\theta(t+s)} dW(t) dW(s).
\end{align*}
The desired result follows.
\end{proof}

It follows from the decomposition (\ref{decomp-num}) along with Proposition %
\ref{two-sided}, we can write 
\begin{equation*}
N_T = \frac{1}{\sigma_{r,\theta}}\left(\frac{Y_{12}(T)}{\sqrt{T}} - \frac{r 
\sqrt{T}}{2 \theta}\right) = \frac{1}{\sigma_{r,\theta}} I^{W}_2(h_T) + 
\frac{1}{\sigma_{r,\theta}} I^{W}_2(g_T) - \frac{r}{4 \theta^2}\frac{%
(1-e^{-2 \theta T})}{\sigma_{r,\theta}\sqrt{T}} - \frac{\sqrt{T}}{%
\sigma_{r,\theta}} \bar{X}_1(T) \bar{X}_2(T).
\end{equation*}
We have for any $x \in \mathbb{R}$ fixed, $\forall \varepsilon>0$, from
Michel and Pfanzagl (1971) \cite{MP} : 
\begin{align}  \label{MP}
\left|F_{N_T}(x) - \phi(x)\right|& \leqslant d_{Kol}\left(N_T, \mathcal{N}%
(0,1)\right) \leqslant d_{Kol}\left(\frac{1}{\sigma_{r,\theta}}
I^{W}_2(h_T), \mathcal{N}(0,1)\right) + \mathbb{P}\left(|Y(T)|>
\varepsilon\right) + \varepsilon,
\end{align}
where 
\begin{equation}  \label{Y}
Y(T) := \frac{1}{\sigma_{r,\theta}} I^{W}_2(g_T) - \frac{r}{4 \theta^2}\frac{%
(1-e^{-2 \theta T})}{\sigma_{r,\theta}\sqrt{T}} - \frac{\sqrt{T}}{%
\sigma_{r,\theta}} \bar{X}_1(T) \bar{X}_2(T).
\end{equation}
To control the first term $\frac{1}{\sigma_{r,\theta}} I^{W}_2(h_T)$, we
will use the following proposition.

\begin{proposition}
There exists $T_0 \geq 0$, such that for all $T \geq T_0$, $\forall x \in 
\mathbb{R}$, 
\begin{equation}  \label{tailI2H}
\left|\mathbb{P}\left(\frac{1}{\sigma_{r,\theta}} I^{W}_2(h_T)< x\right) -
\phi(x) \right| \leqslant \eta(\theta,r) (1+x^{2}) \frac{e^{-\frac{x^2}{2}}}{%
\sqrt{T}}.
\end{equation}
In particular, $T \geq T_0$ 
\begin{equation}  \label{upperkolmoI2H}
d_{Kol} \left(\frac{1}{\sigma_{r,\theta}}I^{W}_2(h_T), \mathcal{N}%
(0,1)\right) \leqslant \frac{2\eta(\theta,r)}{\sqrt{e}} \frac{1}{\sqrt{T}}.
\end{equation}
where the constant $\eta(\theta,r)$ is defined in (\ref{eta}) of Proposition %
\ref{equivalent-FT} of the Appendix.
\end{proposition}

\begin{proof}
The bound (\ref{tailI2H}), is a direct consequence of Proposition \ref%
{equivalent-FT} of the Appendix, the upper bound of the Kolmogorov distance (%
\ref{upperkolmoI2H}) follows using the fact that $\sup\limits_{x \in \mathbb{%
R}} (1+x^2) e^{-\frac{x^2}{2}} = \frac{2}{\sqrt{e}}.$
\end{proof}

For the tail of second chaos random variable $I^{W}_2(g_T)$, we will recall
the following deviation inequality for multiple Wiener integrals, Theorem 2
of \cite{Major}.

\begin{theorem}
\label{tail-major} For any symmetric function $f \in L^{2}([0,T]^{n})$ and $%
x > 0$, we have 
\begin{equation*}
\mathbb{P}\left( \left|I_n(f)\right| > x\right) \leqslant C \exp\left\{- 
\frac{1}{2} \left(\frac{x}{\sqrt{n!} \|f\|_{L^{2}([0,T]^n)}}
\right)^{2/n}\right\},
\end{equation*}
where $I_n(f)$ is the $n$-th Wiener-It\^o integral of $f$ with respect to
the Wiener process and the constant $C> 0$ depends only on the multiplicity
of the integral.
\end{theorem}

A straight forward calculation shows that :

\begin{lemma}
\label{norm-gT} Consider the kernel $g_T$ defined by : $g_T(t,s) = \frac{1}{%
2 \theta} \frac{1}{\sqrt{T}} \left(c_1(r) \mathbf{1}_{[0,T]^2}(t,s) + c_2(r) 
\mathbf{1}(t,s)_{[-T,0]^2}\right) e^{-2 \theta T} e^{\theta |t-s|}.$ Then,
we have 
\begin{equation*}
\| g_T\| _{L^{2}([-T,T])}= \frac{\sqrt{r^{2}+1}}{4\sqrt{2} \theta^2 \sqrt{T}}
(1-e^{-2 \theta T}).
\end{equation*}
\end{lemma}

\begin{proof}
In fact, we have : 
\begin{align*}
\| g_T\|^2 _{L^{2}([-T,T])} & = \int_{-T}^{T} \int_{-T}^{T} g^{2}_T(t,s) dt
ds \\
& = \frac{(c^{2}_1(r)+c^{2}_2(r))}{4 \theta^2} \frac{1}{T} \int_{0}^{T}
\int_{0}^{T} e^{-4\theta T} e^{2 \theta t} e^{2 \theta s} dt ds \\
& = \frac{(c^{2}_1(r)+c^{2}_2(r))}{16 \theta^2} \frac{1}{T} (1-e^{-2 \theta
T})^2 \\
& = \frac{r^2 + 1}{32 \theta^4 } \frac{(1-e^{-2\theta T})^2}{T}.
\end{align*}
\end{proof}

A direct consequence of Theorem \ref{tail-major} and Lemma \ref{norm-gT},
the following bound follows.

\begin{lemma}
\label{Term2} $\forall \varepsilon > 0$ and $T$ large, we have : 
\begin{equation*}
\mathbb{P}\left(\frac{1}{\sigma_{r,\theta}} |I^{W}_2(g_T)| > \varepsilon
\right) \leqslant C \exp\{-\frac{2 \theta^2 \sigma_{r,\theta} \varepsilon}{%
\sqrt{r^2+1}} \sqrt{T}\}.
\end{equation*}
where $C > 0$ is a constant depending only on the multiplicity of the
multiple integral.
\end{lemma}

The remaining term to control is the following : 
\begin{align*}
\mathbb{P}\left(\frac{\sqrt{T}}{\sigma_{r,\theta}} \left|\bar{X}_{1}(T)\bar{X%
}_{2}(T)\right| + \frac{|r|}{4 \theta^2}\times \frac{1}{\sigma_{r,\theta} 
\sqrt{T}}> \frac{\varepsilon}{2}\right) & = \mathbb{P}\left(\frac{\sqrt{T}}{%
\sigma_{r,\theta}} \left|\bar{X}_{1}(T)\bar{X}_{2}(T) \right| > \frac{%
\varepsilon}{2} - \frac{|r|}{4 \theta^2}\times \frac{1}{\sigma_{r,\theta} 
\sqrt{T}}\right)
\end{align*}
In the following, we will denote $\varepsilon(\theta,r) := \frac{\varepsilon%
}{2} - \frac{|r|}{4 \theta^2} \frac{1}{\sigma_{r,\theta} \sqrt{T}}.$ Assume
in the sequel that $\varepsilon > \frac{|r|}{4 \theta^2} \frac{1}{%
\sigma_{r,\theta} \sqrt{T}}$. On the other hand, recall that we have : 
\begin{align*}
X_2(u)& = \left[r \int_{0}^{u}e^{- \theta (u-t)} dW_1(t) + \sqrt{1-r^2}
\int_{0}^{u}e^{- \theta (u-t)} dW_0(t) \right]
\end{align*}
Therefore, 
\begin{align*}
\bar{X}_2(T)& = \frac{1}{T} \int_{0}^{T} X_2(u) du = \frac{r}{T}
\int_{0}^{T} \int_{0}^{u} e^{- \theta(u-v)} dW_1(v) du + \frac{\sqrt{1-r^2}}{%
T} \int_{0}^{T} \int_{0}^{u} e^{- \theta(u-v)} dW_0(v) du \\
& := r \bar{X}_1(T) + \sqrt{1-r^2} \bar{X}_0(T).
\end{align*}
It follows that : 
\begin{equation*}
\mathbb{P}\left(\sqrt{T} |\bar{X}_1(T)\bar{X}_2(T)|> \sigma_{r,\theta}
\varepsilon(\theta,r) \right) \leqslant \mathbb{P} \left(\sqrt{T} \bar{X}%
^2_1(T) > \frac{\sigma_{r,\theta} \varepsilon(\theta,r)}{2|r| }\right) + 
\mathbb{P} \left(\sqrt{T} | \bar{X}_1(T) \bar{X}_0(T)| > \frac{%
\sigma_{r,\theta} \varepsilon(\theta,r)}{2\sqrt{1-r^2}}\right)
\end{equation*}
For the term $\sqrt{T}\bar{X}_1(T) \bar{X}_0(T)$, we have for $i=0,1$ 
\begin{align}
E[\bar{X_{i}}^{2}(T)]& =\frac{1}{T^{2}}\int_{0}^{T}e^{2\theta
u}(\int_{u}^{T}e^{-\theta t}dt)^{2}du  \notag \\
& =\frac{1}{T^{2}}\frac{1}{\theta ^{2}}\int_{0}^{T}(1-e^{-\theta
(T-u)})^{2}du  \notag \\
& = \frac{1}{T} \frac{1}{\theta^2} + O(\frac{1}{T^2}).  \label{normbarX}
\end{align}
Applying Proposition 3.5 of \cite{DEV} to the random variable $\sqrt{T}\bar{X%
}_1(T) \bar{X}_0(T)$, then for $T$ large enough, there exists $\beta(T) := 
\frac{c(\theta)}{\sqrt{T}}>0$, where $c(\theta)$ is a constant depending
only on $\theta$, such that : $E[ e^{\frac{\sqrt{T}\bar{X}_1(T) \bar{X}_0(T)%
}{\beta(T)}}] <2.$ It follows that : 
\begin{equation*}
\mathbb{P} \left(\sqrt{T} | \bar{X}_1(T) \bar{X}_0(T)| > \frac{%
\sigma_{r,\theta} \varepsilon(\theta,r)}{2\sqrt{1-r^2}}\right) \leqslant 2
\exp\left\{- \frac{\sigma_{r,\theta} \varepsilon(\theta,r)}{2 \beta(T) \sqrt{%
1-r^2}} \right\} = 2 \exp\left\{- \frac{\sigma_{r,\theta}
\varepsilon(\theta,r) \sqrt{T}}{2 c(\theta) \sqrt{1-r^2}} \right\}.
\end{equation*}
For the term $\mathbb{P} \left(\sqrt{T} \bar{X}^2_1(T) > \frac{%
\sigma_{r,\theta} \varepsilon(\theta,r)}{2|r| }\right) $, we have for $T$
large : 
\begin{align*}
\mathbb{P} \left(\sqrt{T} \bar{X}^2_1(T) > \frac{\sigma_{r,\theta}
\varepsilon(\theta,r)}{2|r| }\right) & = \mathbb{P} \left(\chi^2(1) > \frac{%
\sigma_{r,\theta} \varepsilon(\theta,r)}{2|r| } \frac{\sqrt{T}}{E[(\sqrt{T}%
\bar{X}_1(T))^2]}\right) \\
& \approx \mathbb{P} \left(\chi^2(1) > \frac{\theta^2 \sqrt{T}%
\sigma_{r,\theta} \varepsilon(\theta,r)}{2|r| } \right) \\
& = \frac{1}{\sqrt{2} \Gamma(1/2)}\int_{ \frac{\theta^2 \sqrt{T}%
\sigma_{r,\theta} \varepsilon(\theta,r)}{2|r| }}^{+\infty} y^{-1/2} e^{-y/2}
dy \leqslant \frac{2}{\sqrt{2}} \exp\left\{- \frac{\sigma_{r,\theta}
\varepsilon(\theta,r) \theta^{2} \sqrt{T}}{8 |r|} \right\}.
\end{align*}
It follows that : 
\begin{equation}  \label{Term3}
\mathbb{P} \left( \sqrt{T} |\bar{X}_1(T)\bar{X}_2(T)|> \sigma_{r,\theta}
\varepsilon(\theta,r)\right) \leqslant 2 \exp \left\{ - \left( \frac{\theta^2%
}{8|r|} \wedge \frac{1}{2 c(\theta)\sqrt{1-r^2}}\right)\sigma_{r,\theta}
\varepsilon(\theta,r)\sqrt{T}\right\}.
\end{equation}
Using the variable $Y(T)$ defined in (\ref{Y}), it follows from Lemma \ref%
{Term2} and (\ref{Term3}) that for all $\varepsilon > \frac{|r|}{4 \theta^2} 
\frac{1}{\sigma_{r,\theta} \sqrt{T}}$, we have : 
\begin{align*}
\mathbb{P} \left(|Y(T)| > \varepsilon \right) + \varepsilon & \leqslant 
\mathbb{P} \left(\frac{1}{\sigma_{r,\theta}} |I^{W}_2(g_T)| > \frac{%
\varepsilon}{2}\right) + \mathbb{P} \left(\frac{\sqrt{T}}{\sigma_{r,\theta}}
|\bar{X}_1(T)\bar{X}_2(T)| > \varepsilon(\theta,r)\right) + \varepsilon \\
& \leqslant (C \vee 2) \times \exp \left\{ - \left( \frac{\theta^2}{8|r|}
\wedge \frac{1}{2 c(\theta)\sqrt{1-r^2}} \wedge \frac{2 \theta^2}{\sqrt{r^2+1%
}}\right)\sigma_{r,\theta} \varepsilon(\theta,r)\sqrt{T}\right\} +
\varepsilon
\end{align*}
We consider the following constants : 
\begin{equation*}
\left\{ 
\begin{array}{ll}
c_{1}(T,r,\theta) = K(\theta,r) \sqrt{T}, C^{\prime} = C \vee 2, &  \\ 
&  \\ 
K(\theta, r) = \left( \frac{\theta^2}{8|r|} \wedge \frac{1}{2 c(\theta)\sqrt{%
1-r^2}} \wedge \frac{2 \theta^2}{\sqrt{r^2+1}}\right)\sigma_{r,\theta} &  \\ 
&  \\ 
c_{2}(T,r,\theta) = \frac{|r|}{4 \theta^2} \frac{1}{\sigma_{r,\theta} \sqrt{T%
}} & 
\end{array}
\right.
\end{equation*}
Then, it's easy to check that the function $\varepsilon \mapsto
g(\varepsilon) := C^{\prime} e^{-c_{1}(T,r , \theta)\left(\frac{\varepsilon}{%
2} - c_2(T,r,\theta)\right) } + \varepsilon$ is convex on $(0,+\infty)$ and
that $\arg \inf\limits_{\varepsilon > 0} g(\varepsilon) = \varepsilon^{*}(T)
= \left(\frac{|r|}{2 \theta^2 \sigma_{r,\theta}} + \frac{2}{K(\theta,r)}
\ln\left(\frac{C^{\prime} K(\theta,r)}{2}\right)\right) \frac{1}{\sqrt{T}} + 
\frac{1}{K(\theta,r)} \frac{\ln(T)}{\sqrt{T}}.$ Therefore for $T$ large, we
get : 
\begin{align*}
\inf\limits_{\varepsilon > 0} g(\varepsilon) = g(\varepsilon^{*}) & = \left[ 
\frac{|r|}{2 \theta^2 \sigma_{r,\theta}} + \frac{2}{K(\theta,r)}
\ln\left(C^{\prime} \frac{K(\theta,r)}{2}\right) + \frac{2}{K(\theta,r)}%
\right] \times \frac{1}{\sqrt{T}} + \frac{1}{K(\theta,r)} \frac{\ln(T)}{%
\sqrt{T}} \\
& \sim \frac{1}{K(\theta,r)} \frac{\ln(T)}{\sqrt{T}}
\end{align*}
It follows from the decomposition (\ref{MP}), that there exits a constant $%
C(\theta, r) >0 $ such that for all $T \geq T_0$ and for all $x \in \mathbb{R%
}$ fixed, we have 
\begin{equation*}
|F_{N_T}(x) - \phi(x)| \leqslant C(\theta,r) \frac{\ln(T)}{\sqrt{T}}.
\end{equation*}
On the other hand, for the normal tails, the following estimate holds for
any $T > \frac{4 \theta^2 c^2_{\alpha}}{r^2}$ : 
\begin{eqnarray}
\left|\mathbf{\phi}\left(\frac{c_{\alpha}}{\sigma_{r,\theta}} - \frac{r\sqrt{%
T}}{2 \theta \sigma_{r,\theta}}\right) - \mathbf{\phi}\left(- \frac{%
c_{\alpha}}{\sigma_{r,\theta}} - \frac{r\sqrt{T}}{2 \theta \sigma_{r,\theta}}%
\right)\right| \leqslant && \sqrt{\frac{2}{\pi}} \frac{c_{\alpha}}{%
\sigma_{r,\theta}} \times \left\{ 
\begin{array}{ll}
e^{-\frac{1}{2}\left(- \frac{c_{\alpha}}{\sigma_{r,\theta}} - \frac{r\sqrt{T}%
}{2 \theta \sigma_{r,\theta}}\right)^2} & \mbox{ if } r \in [-1,0[, \\ 
~~ &  \\ 
e^{-\frac{1}{2}\left( \frac{c_{\alpha}}{\sigma_{r,\theta}} - \frac{r\sqrt{T}%
}{2 \theta \sigma_{r,\theta}}\right)^2} & \mbox{ if } r \in ]0,1].%
\end{array}
\right.
\end{eqnarray}
It follows that for all $T \geq T_0 \vee \frac{4 \theta^2 c^2_{\alpha}}{r^2}$%
, type II error for this test has the following estimate : 
\begin{equation*}
\beta = \mathbb{P}_{H_a}\left[ \left| \frac{Y_{12}(T)}{\sqrt{T}} \right| %
\right]\leqslant C(\theta,\alpha,r) \times \frac{\ln(T)}{\sqrt{T}}.
\end{equation*}
where $C(\theta, \alpha, r) := C(\theta,r) + \sqrt{\frac{2}{\pi}} \frac{%
c_{\alpha}}{\sigma_{r,\theta}},$ which finishes the proof.
\end{proof}

\section{Future directions and and application}

We believe that the strategy behind our testing methodology should be
broadly applicable to many pairs of stationary stochastic processes, and in
particular, to a broad class of stationary Gaussian stochastic processes.
The OU process represents the simplest one in continuous-time modeling.
Extensions to other processes can go in several directions. We present some
ideas about these extensions in this sections first subsection. Its second
subsection covers one particular example of an extension to
infinite-dimensional objects.

\subsection{Future directions}

One may ask whether stationary mean-reverting processes solving non-linear
SDEs, like the Cox-Ingersoll-Ross (CIR) model, will respond to similar
testing with computable rejection regions and provable asymptotic power.
This seems likely in some cases. For instance, the stationary solution to
the CIR SDE is Gamma distributed, which can be construed as a second-chaos
distribution, or can interpolate between such second-chaos distributions,
depending on the shape parameter. The methodology we have developed here
could therefore apply, at the cost of slightly more involved Wiener chaos
computations.

One can ask about discrete-time processes which are also mean-reverting. In
the case of the AR(1) time series with Gaussian innovations, this is known
to be equivalent to an OU process observed at even time intervals. Therefore
a discretisation of this paper's methodology should apply directly in this
case, with increasing horizon, either using methods as in \cite{DEV} or as
in \cite{EHV}. We believe that the same should hold for other time series
models, including any AR($p$) model. However, when $p>1$, AR($p$) is not a
Markov process, and therefore its interpretation as the solution of a SDE is
much less straightforward. The case of AR($p$) with Gaussian innovations
still gives rise to a Gaussian process, and therefore, the same methodology
as in the current paper could apply directly. Unlike in the case of the CIR
model, the price to pay for handling a Gaussian AR($p$) process with $p>1$
lies in the non-explicit nature of the Wiener chaos kernels needed to
represent the solution of AR($p$) as a Gaussian time series, and its
functionals that go into computing the Pearson correlation of a pair of AR($%
p $) processes. This could be technically challenging, though not
conceptually so. The case of time series with non-Gaussian innovations,
particularly heavy-tailed ones, would require a different set of technical
tools, beyond classical Wiener chaos analysis.

This begs the question of whether a more general framework, still based on
Wiener chaos analysis, can be put in place for testing independence of
stationary Gaussian processes in discrete or in continuous time. We believe
there is a limit to how long the Gaussian processes' memory can be while
allowing Gaussian fluctuations for their empirical Pearson correlation, in
the same way that the central limit theorem holds for power and hermite
variations of fractional Gaussian noise (fGn) when the Hurst parameter $H$
is less than some threshold, e.g. $H<3/4$ for quadratic variations of fGn,
but not beyond this point. For quadratic variations of fGn for instance, the
fluctuations are distributed according to the Rosenblatt law, a second-chaos
distribution, which would create practical statistical challenges in
testing. See the Breuer-Major theorem, described for instance in \cite%
{NP-book}, Chapter 7.

While we do not investigate any of these future directions herein, there is
another significant extension which lends itself readily to a
straightforward use of the tools we developed in the previous section, as an
application to infinite-dimensional stochastic processes. We take this up in
the next and final subsection of this paper

\subsection{An application: testing independence with stochastic PDEs\label%
{SPDE}}

We close this paper with a method for testing independence of pairs of
solutions to a basic instance of the stochastic heat equation with additive
noise. As we are about to see, the infinite-dimensional setting is actually
an asset, which allows us to increase the power of independence tests
significantly. Another peculiarity of our test is that the spatial structure
of the underlying noise, or of the SHE's solution, is largely immaterial in
our basic expository framework.

Thus, to place ourselves in a least technical context, consider the
stochastic heat equation on the unit circle (i.e. with periodic boundary
condition on $[-\pi ,\pi ]$) given by :

\begin{equation}
\left\{ 
\begin{array}{ll}
dU(t,x)=\partial _{x,x}^{2}U(t,x)+dW(t,x),\text{ \ }0\leqslant t\leqslant T,%
\text{ \ }x\in \lbrack -\pi ,\pi ]. &  \\ 
~~ &  \\ 
U(0,x)=0, & 
\end{array}%
\right.  \label{SPDE-1}
\end{equation}%
where $W$ is a cylindrical Brownian motions defined on a probability space $%
\left( \Omega ,\mathcal{F},\left\{ \mathcal{F}_{t}\right\} _{t\geq 0},%
\mathbf{P}\right) $. The term cylindrical is interpreted here as meaning
white in space. As is clearly seen from the explicit Fourier expansion of
the unique solution to (\ref{SPDE-1}), given below, this solution is an odd
function which is zero at the boundaries of $[-\pi ,\pi ]$, and thus it is
sufficient to restrict the space variable to $[0,\pi ]$. The following facts
are well known and easy to check directly; we omit references.

\begin{itemize}
\item The Laplacian $\partial _{x,x}^{2}$ has a discrete spectrum $%
v_{k}=k^{2}$, $k\in \mathbb{N}$.

\item The space time (cylindrical) noise $W$ can be written symbolically as

\begin{equation}
\begin{array}{ll}
dW(t,x)=\sum\limits_{k=1}^{+\infty }h_{k}(x)dw_{k}(t) & 
\end{array}
\label{noises}
\end{equation}%
where $\left\{ w_{k},k\geq 1\right\} $, is a family of independent standard
Brownian motions and $\{h_{k},k\geq 1\}$ are the eigenfunctions of $\Delta $%
, given by : 
\begin{equation}
h_{k}(x)=\sqrt{\frac{2}{\pi }}\sin (kx),\ \ k\geq 1,  \label{eigenfunctions}
\end{equation}

\item $\{h_{k},k\geq 1\}$ forms a complete orthonormal system in $%
L^{2}([0,\pi ])$. In this case, using the diagnalization afforded by the
eigen-elements of the Laplacian $\partial _{x,x}^{2}$, the solution $U$ of
equation (\ref{SPDE-1}) can be written as : 
\begin{equation}
U(t,x)=\sum\limits_{k=1}^{+\infty }h_{k}(x)u_{k}(t),  \label{UOU}
\end{equation}%
where the Fourier modes (or coefficients) are given by the solutions to the
SDEs: 
\begin{equation}
du_{k}(t)=-k^{2}u_{k}(t)dt+dw_{k}(t).  \label{OUk}
\end{equation}%
In other words, each Fourier mode $u_{k}$ is an Ornstein-Uhlenbeck process,
as in (\ref{OU}), with rate of mean reversion $\theta $ equal to the
respective eigenvalue $k^{2}$, $k\in \mathbb{N}\backslash \{0\}$, and $%
u_{k}\left( 0\right) =0$. These are the same processes we have studied in
the remainder of this paper.
\end{itemize}

We consider now the projection $U^{N}$ of the solution into $%
H^{N}=Span\{h_{1},...,h_{N}\}$, that is :%
\begin{equation*}
U^{N}(t,x)=\sum\limits_{k=1}^{N}h_{k}(x)u_{k}(t),\ \ i=1,2.
\end{equation*}%
Since the eigenfunctions $h_{k}$ in (\ref{eigenfunctions}) are explicit, we
consider that we have direct access (e.g. via integration against each $%
h_{k} $) to each OU process $u_{k}$, and thus observing $U^{N}\left(
t,x\right) $ over all space and time is equivalent to oberving the set of $N$
independent OU processes $(u_{1},...,u_{N})$. Henceforth, we will abuse the
notation slightly by using $U^{N}$ for the set of these $N$ independent OU
processes.

Let us now assume that we observe two instances (copies) of the random field 
$U$, called $U_{1}$ and $U_{2}$. As mentioned, we thus have access to the
correspondig two copies of the OU processes $u_{k,1}$ and $u_{k,2}$ for any $%
k$. In practice, we will restrict how we keep track of this information by
limiting $k$ to being no greater than $N$. Thus, using the aforementioned
notation, we assume we observe two copies $U_{1}^{N}$ and $U_{2}^{N}$ of the 
$N$ independent OU processes. For each $k$, these processes $u_{k,1}$ and $%
u_{k,2}$ correspond to solutions of (\ref{OU}) with $\theta =k^{2}$ and two
standard Wiener processes $w_{k,1}$ and $w_{k,2}$.

\begin{itemize}
\item \textbf{Our Question (Q}$_{\text{\textbf{N}}}$)\textbf{:} How can we
measure (or test) the dependence or independence between $U_{1}^{N}$ and $%
U_{2}^{N}$ ? More specifically, can we build a statistical test of
independence (or dependence) of the pair $(U_{1}^{N},U_{2}^{N})$? We
consider the following hypotheses

$H_{0}:$ $(U_{1}^{N})$ and $(U_{2}^{N})$ are independent.

Versus

$H_{a}:$ $(U_{1}^{N})$ and $(U_{2}^{N})$ are correlated with correlation $%
r\neq 0$.

\item In order to make this question precise from a modeling perspective,
one must choose how to represent $r=0$ and $r\neq 0$ in these two hypotheses.

\begin{itemize}
\item We represent the first one by assuming that $U_{1}$ and $U_{2}$ are
solutions to (\ref{SPDE-1}) driven respectively by white noises $W_{1}$ and $%
W_{2}$ as in (\ref{noises}), and we require that for every $k$, the OU
processes $u_{k,1}$ and $u_{k,2}$ in (\ref{OUk}) from the representation of $%
U_{1}$ and $U_{2}$ respetively in (\ref{UOU}), are independent. This is
equivalent to requiring that the respective Brownian motions $w_{k,1}$ and $%
w_{k,2}$ in (\ref{OUk}), are independent. We represent the second one by
requiring that there is a fixed $r\neq 0$ which equals the correlation of $%
w_{k,1}$ and $w_{k,2}$, i.e. the same $r\neq 0$ for every $k$
simultaneously, in the respective representations (\ref{UOU}).

\item This question is slightly more involved than the one we treated in the
remainder of this paper, where $N=1$, since now we must ask ourselves
whether this condition relates to asymptotics for the time horizon $T$ tends
to infinity, or whether the number of modes $N$ tends to infinity, or both.
There are other possible options, such as using different numbers of Fourier
modes for each copy, different time horizons (which could also occur if $N=1$%
), and different correlations $r_{k}$ for every $k$. We may also study other
spatial noise structures for white noise in (\ref{SPDE-1}). For a spatial
covariance operator $Q$ for $W$ in \NEG{(}\ref{noises}) is co-diagonalizable
with the Laplacian, this means that we may replace $h_{k}$ in (\ref{noises})
by $\sqrt{\lambda _{k}}h_{k}$ for some sequence of eigenvalues $\lambda _{k}$
for $Q$, , and the solution to (\ref{SPDE-1}) is then the same as in (\ref%
{UOU}) except for replacing $h_{k}$ by $\sqrt{\lambda _{k}}h_{k}$. We can
also consider the case where the SHE (\ref{SPDE-1}) has an initial condition 
$U\left( 0,x\right) =U_{0}\left( x\right) $, which is different from $0$,
which is then easily handled by starting each component $u_{k}$ at tine $0$
at the corresponding value $u_{k,0}=\int h_{k}\left( x\right) U_{0}\left(
x\right) dx$. We will not investigate any of these possibilities, for the
sake of conciseness.
\end{itemize}

\item One may conside the following scenarios :

\begin{enumerate}
\item The number of Fourier modes $N$ is fixed and $T\rightarrow +\infty $ .

\item The time is fixed and $N\rightarrow +\infty $ .

\item Both $T,N\rightarrow +\infty $ .
\end{enumerate}
\end{itemize}

For the sake of conciseness, we consider in detail only option 1 above,
where we fix a number of Fourier modes $N$ and the time horizon $T$ tends to
infinity. See our comments below for the two other cases where the number of
Fourier modes $N$ tends to infinity.

Recall from Question \textbf{(Q}$_{\text{\textbf{N}}}$) that we are looking
for a procedure to reject the null hypothesis $(H_{0})$ that $U_{1}^{N}$ and 
$U_{2}^{N}$ are independent, versus the alternative that each of their $N\,$%
components have a common correlation coefficient $r\neq 0$. If $(H_{0})$
holds, then by integrating $U_{1}^{N}$ and $U_{2}^{N}$ against $h_{k}$, we
get that $u_{k,1}$ and $u_{k,2}$ are independent for every $k\leqslant N$.
The converse holds true of course. Therefore, to reject $(H_{0})$ against
the alternative $(H_{a})$ that each of the $N$ components of $U_{1}^{N}$ and 
$U_{2}^{N}$ have correlation coefficient $r\neq 0$, it is sufficient to
reject the hypothesis $(H_{0,k})$ that $u_{k,1}$ and $u_{k,2}$ are
independent for some $k\leqslant N$, against the alternative $(H_{a,k})$
that $u_{k,1}$ and $u_{k,2}$ have correlation coefficient $r\neq 0$ for that
same value $k$.

Equivalently, the probability of a Type-II error, of failing to reject $%
(H_{0})$ against $(H_{a})$, is the probability of the event that we fail to
reject $(H_{0,k})$ against $(H_{a,k})$ for every $k$.

Working first with a test relative to the empirical correlation $\rho _{k}$
relative to $u_{k,1}$ and $u_{k,2}$, we may then simply use the test
described in Section \ref{Testwithrho}, based on $u_{k,1}$ and $u_{k,2}$,
for every $k\leqslant N$. The Type II error for this test is computed under
the alternative hypothesis. Under this hypothesis (and also under the null),
we know that all the $u_{k}$'s are independent.

Therefore, our Type-II error using the test described in Section \ref%
{Testwithrho} for all $k\leqslant N$ is equal to%
\begin{equation*}
\beta =\prod_{k=1}^{N}\mathbb{P}_{H_{a}}\left[ \sqrt{T}|\rho
_{k}(T)|\leqslant c_{\alpha }\right] .
\end{equation*}%
We may then simply use the upper bound provided by Proposition \ref{estim-t2}%
, and noting that the mean-reversion rate $\theta _{k}$ for $u_{k}$ is simly 
$\theta _{k}=k^{2}$, to obtain%
\begin{eqnarray}
\beta &\leqslant &\prod_{k=1}^{N}\frac{C(k^{2},r,\alpha )}{T^{1/4}}  \notag
\\
&=&T^{-N/4}\prod_{k=1}^{N}C(k^{2},r,\alpha ).  \label{betaSPDE}
\end{eqnarray}%
Since $N$ is fixed in our basic scenario, this leads to a marked improvement
on the rate of converge to 0 of the Type-II error, as soon as $N\geq 2$.
Equivalently, using Corollary \ref{power}, the power of our test, using the
test described in Section \ref{Testwithrho} for all $k\leqslant N$,
converges to $1$ at the rate given in line (\ref{betaSPDE}) above.

The exact same arguments as above, combined with Proposition \ref{estim-num}%
, shows that, if instead, we define our test using the numerator $Y_{12,k}$
of the empirical correlation $\rho _{k}$ of $u_{k,1}$ and $u_{k,2}$ instead
of the full $\rho _{k}$ itself, as definded in Section \ref{Testwithnum},
then the Type-II error $\beta $ is bounded above as%
\begin{equation*}
\beta \leqslant \frac{\left( \ln T\right) ^{N}}{T^{N/2}}%
\prod_{k=1}^{N}C(k^{2},r,\alpha ),
\end{equation*}%
and similarly for the rate of convergence to $1$ of the power of the test.
As before, this improves the rate of by squaring it, except for a
logarithmic correction. However, with only a moderate number $N$ of
components, even with the test based on the $\rho _{k}$'s, we obtain a
relativey fast, polynomial rate of convergence.

We close this section with some comments on what an appropriate value of $N$
might be, as in Scenario 3 defined above, with a view towards a practical
implementation. In such a view, in practice, observations would be in
discrete time, and the ability to compute an approximate value of $\rho _{k}$
based on discrete-time observations of the random field $U\left( t,x\right) $%
, hinges on being able to observe each Fourier component $u_{k}$ at a
sufficiently high rate so that the discrete version of $\rho _{k}$ will be a
good approximation. While this is generally an inoccuous question, when
considering values of $k$ which could be large, when $N$ is large, we need
to keep in mind that the mean rate of reversion of the OU process $u_{k}$ is 
$\theta _{k}:=k^{2}$, which could then be a very large integer. This means
in practice that a faithful observation of the dynamics of this OU process
has to entail a large number of observation points within each time period
where the process is likely to revert back and forth to its mean. Such a
length of time is on the order of $k^{-2}$. How many datapoints are needed
to safely determine a Pearson correlation coefficient depends of course on
how close the alternative $r$ is to $0$, but for values of $r$ wich are not
too small, a rule of thumb is $10^{2}$ as an order of magnitude. With $N=10$%
, which might seem like a moderate value of $N$, this quickly entails at
least $10^{4}$ observation points per unit of time, a mean-rerversion period
length being as small as $10^{-2}$ units of time for and $k$ near $N=10$.
This many datapoints per time unit places values $N\geq 10$ out of the reach
of many applications, as being a high-frequency regime, with significantly
larger $N$ quickly entering the realm of ultra-high frequency. These
comments clearly point us, as a practical matter, to implement our suggested
Fourier-based independence test for solutions of high- or
infinite-dimensional problems like the stochastic heat equation only with a
small number $N$ of components, such as $N=2,3,4$. Since the Type-II error
converges to 0 so quickly for even these moderate values of $N$, there seems
little to be gained for insisting on larger $N$.

A full quantitative analysis of Scenario 3, which depends on realistic
practical parameter estimates, is beyond the scope of this paper, though it
should be straightforward to realize, since the constants in Propositions %
\ref{estim-t2} and \ref{estim-num} are rather explicit functions of $\theta $%
.

We pass on a quantitative analysis of Scenario 2, where $T$ is fixed and $N$
tends to infinity, which is a more complex endeavor, since the main
propositions in this paper are not tailored to asympotics for fixed time
horizon. However, the observation frequency discussion above regarding
Scenario 3 is an indication that asymptotics for $N$ tending to infinity and 
$T$ fixed probably only lead to applicability in ulgtra-high frequency
studies, or analog data with access to continuous-time observations, both of
which are limiting factors.

\section{Appendix}

\begin{lemma}
\label{norm-h} Consider the kernel $h_T$ defined by : $h_T(t,s) = \frac{1}{2
\theta} \frac{1}{\sqrt{T}} \left(c_1(r) \mathbf{1}_{[0,T]^2}(t,s) + c_2(r) 
\mathbf{1}(t,s)_{[-T,0]^2}\right) e^{-\theta |t-s|}$. Then, as $T
\rightarrow +\infty$, we have : 
\begin{equation*}
\| h_T \| _{L^{2}([-T,T]^2)} \longrightarrow \frac{\sqrt{1+r^2}}{2 \sqrt{2}
\theta^{3/2}} = \frac{\sigma_{r,\theta}}{\sqrt{2}}.
\end{equation*}
\end{lemma}

\begin{proof}
We have : 
\begin{align*}
\| h_T \|^2 & = \int_{-T}^{T} \int_{-T}^{T} h^{2}_T(t,s) dt ds \\
& = \frac{c^{2}_1(r)}{4 \theta^2} \frac{1}{T} \int_{0}^{T} \int_{0}^{T} e^{-
\theta |t-s|} dt ds + \frac{c^{2}_2(r)}{4 \theta^2 } \frac{1}{T}
\int_{-T}^{0} \int_{-T}^{0} e^{- \theta |t-s|} dt ds \\
& = \frac{c^{2}_1(r) + c^{2}_2(r)}{4 \theta^2} \int_{0}^{T} \int_{0}^{T}
e^{- \theta |t-s|} dt ds \longrightarrow \frac{c^{2}_1(r) + c^{2}_2(r)}{4
\theta^3} = \frac{1}{4 \theta^3} \left(\frac{r^2 + 1}{2}\right).
\end{align*}
\end{proof}

We recall now Proposition 9.4.1 \cite{NP-book} that we will need in the
sequel.

\begin{proposition}
\label{NP-estimate} Let $N \sim \mathcal{N}(0,1)$ and let $F_n = I_2(f_n), n
\geq 1$, be such that $f_n \in \mathcal{H}^{\odot 2}$. Write $k_p(F_n)$, $p
\geq 1$, for the sequence of the cumulants of $F_n$. Assume that $k_2(F_n) =
E[F^2_n]=1$ for all $n \geq 1$ and that $k_4(F_n) \rightarrow 0$ as $n
\rightarrow +\infty$. If in addition, 
\begin{equation*}
\frac{k_3(F_n)}{\sqrt{k_4(F_n)}} \longrightarrow \alpha, \ \ \text{ \ } \ \ 
\frac{k_8(F_n)}{(k_4(F_n))^2} \longrightarrow 0
\end{equation*}
as $n \rightarrow + \infty$, then for all $z \in \mathbb{R}$ fixed: 
\begin{equation*}
\frac{P(F_n \leqslant z) - P(N \leqslant z)}{\sqrt{k_4(F_n)}}
\longrightarrow \frac{\alpha}{6 \sqrt{2 \pi}} (1-z^2) e^{-\frac{z^2}{2}}, \ 
\text{as} \ n \rightarrow +\infty.
\end{equation*}
In addition, if the constant $\alpha \neq 0$, then there exists a constant $%
c > 0$ and $n_0 \geq 1$ such that : 
\begin{equation}  \label{kolmo-lower}
d_{Kol}(F_n, N) \geq c \sqrt{k_4(F_n)} \ \ \ \forall n \geq n_0.
\end{equation}
\end{proposition}

\begin{proposition}
\label{equivalent-FT} Consider $N \sim \mathcal{N}(0,1)$ and $\tilde{F}_{T}
: = I^{W}_2(\tilde{h}_T)$, where $\tilde{h}_T := \frac{h_T}{\sqrt{2} \| h_T\|%
}$, where the kernel $h_T$ is defined in (\ref{kernels}) and let $%
\delta(t-s) : = \frac{1}{2 \theta} e^{-\theta|t-s|}, \ t,s \in [-T,T]$. We
have $\forall z \in \mathbb{R}$ fixed : 
\begin{equation}  \label{equiv-sd}
P( \tilde{F}_{T} \leqslant z) - P(N \leqslant z) \underset{+\infty}{\sim}
\eta(\theta, r) \times \frac{(1-z^2)}{\sqrt{T}} e^{-\frac{z^2}{2}}.
\end{equation}
where 
\begin{equation}  \label{eta}
\eta(\theta, r) : = \frac{\langle \delta^{*(2)}, \delta \rangle_{\mathcal{L}%
^{2}(\mathbb{R})} }{\sqrt{ \pi} } \frac{2^2 \theta^{9/2} r (3-r^2)}{%
(1+r^2)^{3/2} }.
\end{equation}
\end{proposition}

\begin{proof}
Applying Proposition \ref{equiv-cumulants} to the random variable $\tilde{F}%
_T$, we get 
\begin{equation*}
k_3(\tilde{F}_T) \underset{+\infty}{\sim} \frac{\langle \delta^{*(2)},
\delta\rangle_{\mathcal{L}^{2}(\mathbb{R})} (c^{3}_1(r) + c^{3}_2(r)) 2^{6}
\theta^{9/2}}{T^{1/2} (1+r^2)^{3/2}}
\end{equation*}
and 
\begin{equation}  \label{k4I2}
k_4(\tilde{F}_{T}) \underset{+\infty}{\sim} \frac{\langle \delta^{*(3)},
\delta\rangle_{\mathcal{L}^{2}(\mathbb{R})} (c^{4}_1(r) + c^{4}_2(r)) 2^{7}
\theta^{6} 3!}{T (1+r^2)^2}.
\end{equation}
Consequently, and based on remark \ref{neq0}, $c^{4}_1(r) + c^{4}_2(r) \neq
0 $, the following convergence holds: 
\begin{equation*}
\frac{k_3(\tilde{F}_T)}{\sqrt{k_4(\tilde{F}_T)}} \underset{T \rightarrow
+\infty}{\longrightarrow} \alpha(\theta,r ) : = \frac{\langle \delta^{*(2)},
\delta \rangle_{\mathcal{L}^{2}(\mathbb{R})} \theta^{3/2}}{\sqrt{3} \sqrt{%
|\langle \delta^{*(3)}, \delta\rangle_{\mathcal{L}^{2}(\mathbb{R})}|}} \frac{%
r\sqrt{2} (3-r^2)}{\sqrt{1+r^2} \sqrt{c^{4}_1(r) + c^{4}_2(r)}} \neq 0.
\end{equation*}
For the eight cumulant of $\tilde{F}_{T}$, we have 
\begin{equation*}
k_8(\tilde{F}_{T}) \underset{+\infty}{\sim} \frac{\langle \delta^{*(7)},
\delta \rangle_{\mathcal{L}^{2}(\mathbb{R})} (c^{8}_1(r) + c^{8}_2(r))2^{15}
7! \times \theta^{12}}{T^{3} (1+r^2)^2}.
\end{equation*}
It follows that : 
\begin{equation*}
\frac{k_8(\tilde{F}_T)}{(k_4(\tilde{F}_{T}))^2} \underset{+\infty}{\sim} 
\frac{\langle \delta^{*(7)}, \delta \rangle_{\mathcal{L}^{2}(\mathbb{R})} }{
(\langle \delta^{*(3)}, \delta \rangle_{\mathcal{L}^{2}(\mathbb{R})})^2 }
\times \frac{ (c^{8}_1(r) + c^{8}_2(r))}{ (c^{4}_1(r) + c^{4}_2(r))} \frac{%
\theta^6 2^{8} \times 7!}{3!} \frac{1}{T}.
\end{equation*}
It follows that 
\begin{equation*}
\frac{k_8(\tilde{F}_T)}{(k_4(\tilde{F}_{T}))^2} \underset{T \rightarrow
+\infty}{\longrightarrow} 0.
\end{equation*}
Therefore applying Proposition \ref{NP-estimate}, we get $\forall z \in 
\mathbb{R}$ fixed, 
\begin{equation*}
\frac{P( \tilde{F}_{T} \leqslant z) - P(N \leqslant z)}{\sqrt{k_4(\tilde{F}%
_{T})}} \underset{T \rightarrow +\infty}{\longrightarrow} \frac{%
\alpha(\theta,r)}{6 \sqrt{2 \pi}} (1-z^2) e^{-\frac{z^2}{2}}.
\end{equation*}
Consequently, $\forall z \in \mathbb{R}$ fixed : 
\begin{align}  \label{equiv-sd}
P( \tilde{F}_{T} \leqslant z) - P(N \leqslant z) & \underset{+\infty}{\sim} 
\frac{\langle \delta^{*(2)}, \delta \rangle_{\mathcal{L}^{2}(\mathbb{R})} }{%
\sqrt{ \pi} \sqrt{T}} \frac{2^2 \theta^{9/2} r (3-r^2)}{(1+r^2)^{3/2} }%
(1-z^2) e^{-\frac{z^2}{2}}  \notag \\
& \underset{+\infty}{\sim} \eta(\theta,r) (1-z^2) e^{-\frac{z^2}{2}}.
\end{align}
which finishes the proof.
\end{proof}

\begin{proposition}
\label{equiv-cumulants} Consider $\tilde{F}_{T} : = I^{W}_2(\tilde{h}_T)$,
where $\tilde{h}_T := \frac{h_T}{\sqrt{2} \| h_T\|}$, where the kernel $h_T$
is defined in (\ref{kernels}) and let $\delta(t-s) : = \frac{1}{2 \theta}
e^{-\theta|t-s|}, \ t,s \in [-T,T]$. Then, 
\begin{equation*}
\forall p\geq 3, \ \ k_p\left(\tilde{F}_T\right) \sim \frac{\langle
\delta^{*(p-1)}, \delta\rangle_{\mathcal{L}^{2}(\mathbb{R})} 2^{2p-1} (p-1)!
(c^{p}_1(r) + c^{p}_2(r)) \theta^{3p/2}}{T^{\frac{p}{2}-1} (1+r^2)^{p/2}}.
\end{equation*}
Let $\delta^{*(p)}$ denotes the convolution of $\delta$ p times defined as $%
\delta^{*(p)} = \delta^{*(p-1)}* \delta$, $p \geq 2$, $\delta^{*(1)} =
\delta $ where $*$ denotes the convolution between two functions $(f*g)(x) =
\int_{\mathbb{R}} f(x-y)g(y) dy$.
\end{proposition}

\begin{proof}
The proof of this Proposition is an extension of the proof of Proposition
7.3.3. of \cite{NP-book} for the continuous time setting. \newline
Recall that when $F= I_2(f)$, $f \in \mathcal{H}^{\odot 2}$, then the
sequence of cumulants of $F$, $k_p(F)$ for all $p \geq 2$ can be computed as
follows : 
\begin{equation*}
\forall p \geq 2, \ \ \ k_p(F) = 2^{p-1} \times (p-1)! \times \langle
f\otimes^{(p-1)}_1 f, f \rangle_{\mathcal{H}^{\otimes 2}}
\end{equation*}
where the sequence of kernels $\{ f\otimes^{(p)}_1 f, p \geq 1\}$ is defined
as follows $f\otimes^{(1)}_1 f = f$ and for $p \geq 2$, $f\otimes^{(p)}_1 f
= (f\otimes^{(p-1)}_1 f)\otimes_1 f$. Let $p \geq 3$, $\tilde{F}_{T} = I_2(%
\tilde{h}_{T})$, where $\tilde{h}_{T} = \frac{h_T}{\sqrt{2} \| h_T\|}$, then 
\begin{align*}
k_p(\tilde{F}_{T}) & = 2^{p-1} \times (p-1)! \times \langle \tilde{h}%
_T\otimes^{(p-1)}_1 \tilde{h}_T, \tilde{h}_T \rangle_{L^2{([-T,T]^2)}} \\
& = 2^{p-1} \times (p-1)! \times \int_{[-T,T]^2} (\tilde{h}%
_T\otimes^{(p-1)}_1 \tilde{h}_T)(u_1,u_2) \tilde{h}_T(u_1,u_2) du_1 du_2 \\
& = 2^{p-1} \times (p-1)! \times \int_{[-T,T]^2} ((\tilde{h}%
_T\otimes^{(p-2)}_1 \tilde{h}_T) \otimes_{1} \tilde{h}_T) (u_1,u_2) \tilde{h}%
_T(u_1,u_2) du_1 du_2 \\
& = 2^{p-1} \times (p-1)! \times \int_{[-T,T]^3} (\tilde{h}%
_T\otimes^{(p-2)}_1 \tilde{h}_T)(u_1,u_3) \tilde{h}_T(u_3,u_2) \tilde{h}%
_T(u_1,u_2) du_1 du_2 du_3 \\
& = \vdots \\
& = 2^{p-1} \times (p-1)! \times \int_{[-T,T]^p} \tilde{h}_T(u_p,u_1) \times%
\tilde{h}_T(u_p,u_{p-1}) \times ... \times \tilde{h}_T(u_3,u_2) \tilde{h}%
_T(u_1,u_2) du_1 du_2...du_p \\
& = \frac{ 2^{p-1} c^{p}_1(r) \times (p-1)!}{(\sqrt{T} \sqrt{2} \|h_T\|)^{p}}
\times \int_{[0,T]^p} \delta(u_p-u_1) \times \delta(u_p - u_{p-1}) \times
... \times \delta(u_3-u_2) \delta(u_2-u_1) du_1 du_2 ... du_p \\
& \ \ \ \ + \frac{ 2^{p-1} c^{p}_2(r) \times (p-1)!}{(\sqrt{T} \sqrt{2}
\|h_T\|)^{p}} \times \int_{[-T,0]^p} \delta(u_p-u_1) \times \delta(u_p -
u_{p-1}) \times ... \times \delta(u_3-u_2) \delta(u_2-u_1) du_1 du_2 ... du_p
\\
& = \frac{ 2^{p-1} (c^{p}_1(r) + c^{p}_2(r)) \times (p-1)!}{(\sqrt{T} \sqrt{2%
} \|h_T\|)^{p}} \int_{[0,T]^p} \delta(u_p-u_1) \times \delta(u_p - u_{p-1})
\times ... \times \delta(u_3-u_2) \delta(u_2-u_1) du_1 du_2 ... du_p.
\end{align*}
Using the change of variable $v_i = u_i-u_1$, $i \geq 2$, then : 
\begin{align*}
k_p(\tilde{F}_{T}) & = \frac{ 2^{p-1} (c^{p}_1(r) + c^{p}_2(r))\times (p-1)!%
}{(\sqrt{T} \sqrt{2} \|h_T\|)^{p}} \int_{0}^{T}
\int_{-u_1}^{T-u_1}...\int_{-u_1}^{T-u_1} \delta(v_p)
\delta(v_p-v_{p-1})\times ...\times \delta(v_3-v_2) dv_2 dv_3...dv_p du_1
\end{align*}
On the other hand, by dominated convergence theorem, we have 
\begin{align*}
& \frac{1}{T} \int_{-u_1}^{T-u_1}...\int_{-u_1}^{T-u_1} \delta(v_p)
\delta(v_p-v_{p-1})\times ...\times \delta(v_3-v_2) dv_2 dv_3...dv_p du_1 \\
& = \frac{1}{T} \int_{\mathbb{R}^{p-1}} \int_{0 \vee -v_2 \vee...\vee
-v_p}^{T \wedge(T-v_2) \wedge ...\wedge (T-v_p)} du_1 \delta(v_p)
\delta(v_p-v_{p-1}) ...\delta(v_{3}-v_2) \delta(v_2) dv_p ... dv_2 \mathbf{1}%
_{\{|v_p|< T, ..., |v_2|< T \}} dv_2... dv_p \\
& = \int_{\mathbb{R}^{p-1}} \delta(v_p) \delta(v_p-v_{p-1})
...\delta(v_{3}-v_2) \delta(v_2) \left[ 1 \wedge \left(1 - \frac{v_2 \vee
v_3 \vee ... \vee v_p }{T}\right) - 0 \vee \frac{v_2 \wedge v_3 \wedge...
\wedge v_p}{T} \right] \mathbf{1}_{\{|v_p|< T, ..., |v_2|< T \}} dv_2...dv_p
\\
& \underset{T \rightarrow +\infty}{\longrightarrow} \int_{\mathbb{R}^{p-1}}
\delta(v_p) \delta(v_p-v_{p-1}) ...\delta(v_{3}-v_2) \delta(v_2) dv_2...
dv_p = \langle \delta^{*(p-1)}, \delta \rangle_{\mathcal{L}^{2}(\mathbb{R})}
< + \infty.
\end{align*}
For the assertion $\langle \delta^{*(p-1)}, \delta \rangle_{\mathcal{L}^{2}(%
\mathbb{R})} < + \infty$, we need to check that the function $\delta \in 
\mathcal{L}^{\frac{p}{p-1}}(\mathbb{R})$, because in this case the $\mathcal{%
L}^{p}$ norm of the $(p-1)$ convolution is finite $\|\delta^{*(p-1)} \|_{%
\mathcal{L}^{p}(\mathbb{R})} < + \infty$. \newline
In fact by Holder's inequality then Young inequality, we get 
\begin{align*}
|\langle \delta^{*(p-1)}, \delta \rangle_{\mathcal{L}^{2}(\mathbb{R})}| &
\leqslant \| \delta \|_{\mathcal{L}^{\frac{p}{p-1}}(\mathbb{R})} \times \|
\delta^{*(p-1)}\|_{\mathcal{L}^{p}(\mathbb{R})} = \| \delta \|_{\mathcal{L}^{%
\frac{p}{p-1}}(\mathbb{R})} \times \| \delta^{*(p-2)}* \delta\|_{\mathcal{L}%
^{p}(\mathbb{R})} \\
& \leqslant \| \delta \|^2_{\mathcal{L}^{\frac{p}{p-1}}(\mathbb{R})} \times
\| \delta^{*(p-2)}\|_{\mathcal{L}^{p/2}(\mathbb{R})} \\
& \leqslant \| \delta \|^3_{\mathcal{L}^{\frac{p}{p-1}}(\mathbb{R})} \times
\| \delta^{*(p-3)}\|_{\mathcal{L}^{p/3}(\mathbb{R})} ... \leqslant \| \delta
\|^p_{\mathcal{L}^{\frac{p}{p-1}}(\mathbb{R})}
\end{align*}
It remains to check that $\delta \in \mathcal{L}^{\frac{p}{p-1}}(\mathbb{R})$%
, we have 
\begin{align*}
\| \delta \|^{p/(p-1)}_{\mathcal{L}^{\frac{p}{p-1}}(\mathbb{R})} & = \frac{1%
}{2^{\frac{p}{p-1}} \theta^{\frac{p}{p-1}}} \int_{\mathbb{R}} e^{-\frac{p}{%
p-1} \theta |u|} du = \frac{p-1}{p\times 2^{\frac{p}{2} - 2} \times \theta^{%
\frac{2p-1}{p-1}}} < + \infty.
\end{align*}

\begin{remark}
\label{neq0} Notice that the constant $c^{p}_1(r) + c^{p}_2(r) \neq 0$, $%
\forall p \geq 3$. In fact, we can easily check that : 
\begin{eqnarray}
\left\{ 
\begin{array}{ll}
c_1(r) = 0 & \iff r = -\frac{1}{\sqrt{3}} \\ 
&  \\ 
c_2(r) = 0 & \iff r = \frac{1}{\sqrt{3}}.%
\end{array}
\right.
\end{eqnarray}
It follows that :

\begin{itemize}
\item If $r = \frac{1}{\sqrt{3}}$, then $c^{p}_1(\frac{1}{\sqrt{3}}) +
c^{p}_2(\frac{1}{\sqrt{3}}) = \left(\frac{\sqrt{2}}{\sqrt{3}}\right)^{p}
\neq 0.$

\item If $r = -\frac{1}{\sqrt{3}}$, then $c^{p}_1(-\frac{1}{\sqrt{3}}) +
c^{p}_2(-\frac{1}{\sqrt{3}}) = (-1)^{p} \left(\frac{\sqrt{2}}{\sqrt{3}}%
\right)^{p} \neq 0.$
\end{itemize}
\end{remark}

Consequently, the following convergence holds: 
\begin{equation*}
k_p(\tilde{F}_{T}) \times \frac{T^{\frac{p}{2}-1} 2^{p/2} \|h_T\|^p}{2^{p-1}
(c^{p}_1(r) + c^{p}_2(r)) (p-1)! } \underset{T \rightarrow +\infty}{%
\longrightarrow} \langle \delta^{*(p-1)}, \delta \rangle_{\mathcal{L}^{2}(%
\mathbb{R})}
\end{equation*}
Finally, by Lemma \ref{norm-h}, we have : $2^{p/2}\|h_T\|^{p} \underset{%
\infty}{\sim} \frac{(1+r^2)^{p/2}}{2^{p} \theta^{3p/2}}$, the desired result
follows.
\end{proof}

\end{document}